%% file: qctvg7.tex
\documentclass[12pt]{amsart}
\usepackage{amscd}
\usepackage{amssymb}
\usepackage{graphicx}
\usepackage{color}
\usepackage[centering,text={15.5cm,22cm},
		marginparwidth=20mm]{geometry}
\usepackage{srcltx}
\usepackage[]{hyperref}

\usepackage{mathrsfs}

\usepackage[all]{xy}

\newtheorem{theorem}{Theorem}[section]
\newtheorem{lemma}[theorem]{Lemma}
\newtheorem{proposition}[theorem]{Proposition}

\theoremstyle{definition}

\theoremstyle{remark}

\numberwithin{equation}{section}
\numberwithin{figure}{section}

\newcommand{\ZZ} {\mathbb{Z}}

\newcommand{\RR} {\mathbb{R}}
\newcommand{\CC} {\mathbb{C}}

\newcommand{\PP} {\mathbb{P}}
\renewcommand{\AA} {\mathbb{A}}

\newcommand {\shM}  {\mathcal{M}}

\newcommand {\shP}  {\mathcal{P}}

\newcommand {\foM}  {\mathfrak{M}}

\newcommand {\foX}  {\mathfrak{X}}

\newcommand {\TT} {{\mathsf{T}}}

\newcommand {\GL}  {\operatorname{GL}}

\newcommand {\Hom}  {\operatorname{Hom}}

\newcommand {\out}  {\mathrm{out}}

\def\mydate{\ifcase\month \or January\or February\or March\or
April\or May\or June\or July\or August\or September\or October\or 
November\or December\fi \space\number\day,\space\number\year}

\begin{document}

\title[Quivers, curves, and the tropical vertex]{Quivers, curves,
and the tropical vertex}

\author{M. Gross} \address{Department of Mathematics, UCSD,
La Jolla, CA 92093, USA}
\email{mgross@math.ucsd.edu}

\author{R. Pandharipande}\address{Department of Mathematics, Princeton 
Univerity, Princeton, NJ 08544, USA}
\email{rahulp@math.princeton.edu}

\date{September 2009}

\begin{abstract}
Elements of the tropical vertex group are formal families
of symplectomorphisms of the 2-dimensional algebraic
torus. Commutators in the group are related
to Euler characteristics of the moduli spaces
of quiver representations and the Gromov-Witten theory
of toric surfaces. 
After a short survey of the subject (based on
lectures of Pandharipande at the 2009 {\em Geometry
summer school} in Lisbon), we 
prove new results about the rays and symmetries of scattering
diagrams of commutators (including previous conjectures by 
Gross-Siebert and Kontsevich).
Where possible, we present both the quiver and Gromov-Witten perspectives.
\end{abstract}

\maketitle
\setcounter{tocdepth}{1}
\tableofcontents


\section*{Introduction}

In Sections 1-3 of the paper,
we survey the recently discovered
relationship of three mathematical structures:
\begin{enumerate}
\item[(i)]
Euler characteristics of the moduli spaces
of quiver representations,
\item[(ii)] Gromov-Witten counts of rational curves on
toric surfaces, 
\item[(iii)] Ordered product factorizations of commutators in
the tropical vertex group.
\end{enumerate}
The tropical vertex group (iii) first arose in the work of Kontsevich and
Soibelman \cite{ks} and plays a significant role
in the program of \cite{GS}. A connection of the tropical
vertex group to (i) has been proven by Reineke \cite{R2}
using wall-crossing ideas. 
 A connection to (ii) is proven
in \cite{GPS}. 
Our aim here is to present the shortest path to the simplest cases
of the results. Lengthier treatments can be found in the
original references.

The definition and basic properties of the tropical vertex group are 
reviewed in Section 1. Reineke's result is Theorem 1 of Section 2.
The formula of \cite{GPS} relating commutators in the tropical
vertex group to rational curve counts is Theorem 2 of Section 3.
Put together, Theorems 1 and 2 yield a suprising 
equivalence between curve counts on toric surfaces
and Euler characteristics of moduli spaces of quiver representations. 
The equivalence is stated in Corollary 3
without any reference to the tropical vertex group.

In Section 4, we  address the question of
which slopes occur in the ordered product factorizations of
commutators (iii). In the language of (i), the question 
asks which slopes are achieved by semistable representations
of particular quivers. In Theorem 5,
we find necessary conditions from the perspective of (ii)
using the 
classical
geometry of curves on surfaces. The result includes all the
previous conjectures on scattering patterns as
special cases.

Symmetries of the  commutator factorizations are
proven in Theorem 7 of
 Section 5. From the point of view of curve counting,
the symmetries are obtained by transformations of 
blown-up toric surfaces. On the quiver side, the symmetries
are a consequence of well-known reflection functors.
Further directions in the subject are suggested in Section \ref{furr}

\section{The tropical vertex group}

\subsection{Automorphisms of the torus}
The 2-dimensional complex torus has very few
automorphisms 
$$\theta: \CC^* \times \CC^* \rightarrow \CC^* \times \CC^*$$
as an algebraic group. Since $\theta$ must take each component
$\CC^*$ to a
$1$-dimensional
subtorus, 
$$\text{Aut}_\CC^{\text{Gr}}(\CC^* \times \CC^*) 
\stackrel{\sim}{=} \GL_2(\mathbb{Z}).$$
As a complex algebraic variety, $\CC^* \times \CC^*$ has, in
addition, only the automorphisms obtained by the translation action
on itself,{\footnote{We leave the elementary proof
to the reader. An argument can be found by using
the characterization
$$\phi(z) = \lambda \cdot z^k\, \ \ \ \  \lambda\in \CC^*, \ k \in \mathbb{Z}$$
of all algebraic maps $\phi: \CC^* \rightarrow \CC^*$.
}} 
$$1 \rightarrow \CC^* \times \CC^*
 \rightarrow 
\text{Aut}_{\CC}(\CC^* \times \CC^*) \rightarrow
\text{Aut}_\CC^{\text{Gr}}(\CC^* \times \CC^*)
  \rightarrow 1.$$

A much richer algebraic structure appears if formal 1-parameter
families of automorphisms of $\CC^*\times \CC^*$ 
are considered,
$$A = \text{Aut}_{\CC[[t]]}(\CC^* \times \CC^* \times \text{Spec}(\CC[[t]])).$$
Let $x$ and $y$ be the coordinates of the two factors
of $\CC^* \times \CC^*$. Then,
$$\CC^* \times \CC^* = \text{Spec}(\CC[x,x^{-1},y,y^{-1}]).$$
We may alternatively view $A$ as
a group of algebra automorphisms,
$$A = \text{Aut}_{\CC[[t]]} (\CC[x,x^{-1},y,y^{-1}][[t]]).$$

Nontrivial elements of $A$ are easily found.
Let $(a,b)\in \mathbb{Z}^2$ be a nonzero vector, and let
$f \in \CC[x,x^{-1},y,y^{-1}][[t]]$ be a function of the form
$$f= 1 + t x^a y^b\cdot g( x^ay^b,t), \ \ \ \ g(z,t) \in \CC[z][[t]].$$
We specify the values of an automorphism on $x$ and $y$ by
\begin{equation}\label{vrzza}
\theta_{(a,b),f\ }(x) = x \cdot f^{-b}, \ 
\ \theta_{(a,b),f\ }(y)= y \cdot f^a \ \ .
\end{equation}
The assignment \eqref{vrzza} extends uniquely to determine
an element $\theta_{(a,b),f} \in A$.
The inverse is obtained by inverting $f$,
$$\theta_{(a,b),f}^{-1} = \theta_{(a,b),f^{-1}} \ . $$

\subsection{Tropical vertex group}
The tropical vertex group $H\subset A$ is the completion
with respect to the maximal ideal $(t) \subset \CC[[t]]$
of the subgroup 
generated by {\em all} elements of the form $\theta_{(a,b),f}$.
In particular, infinite products are well-defined
in $H$ if only finitely many terms are
nontrivial mod $t^k$ (for every $k$).
A more natural characterization of $H$
via the associated Lie algebra may be found in Section 1.1 of
\cite{GPS}.

The torus $\CC^* \times \CC^*$ has a standard 
holomorphic symplectic form given by 
$$\omega = \frac{dx}{x} \wedge \frac{dy}{y}.$$
Let $S\subset A$ be the subgroup of automorphisms
preserving $\omega$,
$$S = \{ \ \theta \in A\ | \ \theta^*(\omega)= \omega \ \}.$$

\begin{lemma} $H\subset S$. \end{lemma}

\begin{proof}
The result is obtained from a direct calculation. Let
$$\widetilde{x} = x f^{-b}, \ \ \ \ \ \widetilde{y}= y f^{a}.$$
From the equations
$$\frac{d \widetilde{x}}{\widetilde{x}}= \frac{dx}{x}
-\frac{b f_x}{f} dx -  \frac{bf_y}{f} dy, \ \ \ \ \  
\frac{d \widetilde{y}}{\widetilde{y}}= \frac{dy}{y}
+\frac{a f_y}{f} dy + \frac{af_x}{f} dx,$$
we conclude $\theta^*_{(a,b),f}(\omega)= \omega$
if 
$$\frac{a f_y}{xf} =  \frac{bf_x}{yf}.$$
The latter 
follows from the dependence of $f$ on $x$ and $y$
only through $x^ay^b$.
\end{proof}

A slight variant of the tropical vertex group
 $H$ first arose in the study of affine
structures by Kontsevich and Soibelman in \cite{ks}.
Further development, related to mirror symmetry and tropical geometry, can be found in 
\cite{GS}. 
Recently, the tropical vertex group has played a role in 
wall-crossing formulas for counting invariants in derived
categories \cite{ks2}.

\subsection{Commutators}
The first question  we can ask about the tropical
vertex group is to find
a formula for the commutators of the generators.
The answer is related to Euler characteristics of
moduli spaces of quiver representations and to
 Gromov-Witten counts of rational curves on
toric surfaces.
The simplest nontrivial cases to consider are the commutators of 
the elements
$$S_{\ell_1}=\theta_{(1,0),(1+tx)^{\ell_1}} \ \  \text{and} \ \ 
T_{\ell_2}=\theta_{(0,1), (1+ty)^{\ell_2}}$$
where $\ell_1,\ell_2 >0$.
By an elementary result of \cite{ks} reviewed in Section 1.3
of \cite{GPS}, there exists a unique
factorization
\begin{equation}\label{jjtr}
    T_{\ell_2}^{-1} \circ S_{\ell_1} \circ T_{\ell_2}  \circ S_{\ell_1}^{-1} =
\stackrel{\rightarrow} \prod \theta_{(a,b),f_{a,b}}\ 
\end{equation}
where the product on the right is over {\em all} primitive
vectors $(a,b)\in \mathbb{Z}^2$ lying strictly in the first quadrant.{\footnote{A vector
 $(a,b)$ is primitive if it is not divisible in $\mathbb{Z}^2$. 
Primitivity implies $(a,b)\neq (0,0)$.
Strict inclusion in
the first  quadrant is equivalent to $a>0$ and $b>0$.}}${}^{,}${\footnote{Here 
and throughout the paper, we drop the dependence of $f_{a,b}$
 upon $(\ell_1,\ell_2)$ for notational convenience.}}
The order is determined by 
increasing slopes of the vectors $(a,b)$.
The product \eqref{jjtr} is very often infinite, but
always has only finitely many nontrivial terms mod $t^k$ (for every $k$).
%
The question is what are the functions $f_{a,b}$ 
associated to the slopes?

\subsection{Examples} \label{exxc}
The easiest example is $\ell_1=\ell_2=1$. The
formula
\[
T_1^{-1}\circ S_1\circ T_1\circ S_1^{-1}=\theta_{(1,1),1+t^2xy}
\]
can  be directly checked
by hand. We will display the information by drawing rays of slope
$(a,b)$ in the first quadrant for every term appearing on the
right-hand side. Each ray should be thought of as labelled with
a function, see Figure \ref{diagram11}.

\begin{figure}
\input{diagram11.pstex_t}
\caption{}
\label{diagram11} 
\end{figure}

For $\ell_1=\ell_2=2$, we already
have a much more complicated expansion, 
\begin{eqnarray*}
T_2^{-1}\circ S_2\circ T_2\circ S_2^{-1}&=&\theta_{(1,2),(1+t^3xy^2)^2}
\,\circ \,\theta_{(2,3),(1+t^5x^2y^3)^2}\,\circ\,\theta_{(3,4),(1+t^7x^3y^4)^2}
\, \circ \cdots\\
&&\ \ \ \ \ \ \ \ \  \ \ \ \ \ \ \ \ 
\ \ \ \   \ \circ\,\ \theta_{(1,1),1/(1-t^2xy)^4}\,\circ\\
&&\cdots\circ\,\theta_{(4,3),(1+t^7x^4y^3)^2}
\,\circ\,
\theta_{(3,2),(1+t^5x^3y^2)^2}
\,\circ\,
\theta_{(2,1),(1+t^3x^2y)^2}.
\end{eqnarray*}
The values of $(a,b)$ which occur are of the form $(k,k+1)$ and
$(1,1)$ and $(k+1,k)$ for all $k\ge 1$. We depict the slopes
occuring by rays in the first quadrant as in
Figure \ref{diagram22}. Ideally, we would label each ray $\RR_{\ge 0}(a,b)$
with the function $f_{a,b}$, however the diagram would become
too difficult to draw. Here 
\begin{eqnarray*}
f_{1,1}&=&1/(1-t^2xy)^4\\
f_{k,k+1}&=&(1+t^{2k+1}x^ky^{k+1})^2\\
f_{k+1,k}&=&(1+t^{2k+1}x^{k+1}y^{k})^2 \ .
\end{eqnarray*}

\begin{figure}
\input{diagram22.pstex_t}
\caption{}
\label{diagram22} 
\end{figure}

The case $\ell_1=\ell_2=3$ becomes
still more complex,  illustrated in Figure \ref{diagram33}.
Extrapolating from calculations, we find
rays with primitives
$$(a,b)=(3,1),\ (8,3),\ (21,8),\ \ldots$$ converging
to the ray of slope $(3-\sqrt{5})/2$ and rays with primitives
$$(a,b)=(1,3),\ (3,8),\ (8,21),\ \ldots$$ 
converging to the ray of slope
$(3+\sqrt{5})/2$. Meanwhile, all rays with rational slope between
$(3-\sqrt{5})/2$ and $(3+\sqrt{5})/2$ appear to occur.

We do not know closed forms for the
functions associated to each ray. However,
Gross conjectured  the function attached to the line of
slope $1$ in 
 Figure \ref{diagram33}
is
\begin{equation}\label{vvbb}
\left(\sum_{k=0}^{\infty} {1\over 3k+1}\begin{pmatrix} 4k\\ k\end{pmatrix}
t^{2k}x^ky^k\right)^9.
\end{equation}

\begin{figure}
\input{diagram33.pstex_t}
\caption{}
\label{diagram33} 
\end{figure}

Finally, consider the asymmetric case  $(\ell_1,\ell_2)=(2,3)$. 
We again appear to obtain a discrete series of rays and
a cone in which all rays  occur. We find rays with
primitives
$$(a,b)=(2,1),\ (5,2),\ (8,5),\ (19,12),\ \ldots$$ converging to a ray
of slope $(3-\sqrt{3})/2$ and rays with primitives 
$$(a,b)=(1,3),\ (2,5),\ (5,12),\
(8,19),\ \ldots$$ 
converging to a ray of slope $(3+\sqrt{3})/2$. 
All rays with rational slope in between these two quadratic
irrational slopes seem to appear. The function attached to the ray of slope
$1$ appears to be
\[
\left(\sum_{k=0}^{\infty} {1\over k+1}\begin{pmatrix}2k\\ k\end{pmatrix}
t^{2k}x^ky^k\right)^6.
\]
Inside the exponential is the generating series for Catalan numbers.

\vspace{10pt}
\noindent{\bf Conjecture.}
{\em For arbitrary $(\ell_1,\ell_2)$, the function attached
to the ray of slope $1$ is
\begin{equation}\label{nnyy}
\left(\sum_{k=0}^{\infty}{1\over (\ell_1\ell_2-\ell_1-\ell_2)k+1}
\begin{pmatrix} (\ell_1-1)(\ell_2-1)k\\ k\end{pmatrix}
t^{2k}x^ky^k
\right)^{\ell_1\ell_2}.
\end{equation}}

\vspace{10pt} 
The above conjecture specializes to the series \eqref{vvbb}
in the 
$(\ell_1,\ell_2)=(3,3)$ case.
The specialization of \eqref{nnyy} to 
 $\ell_1=\ell_2$ was conjectured  by Kontsevich  (motivated
by \eqref{vvbb}) and 
proved by Reineke in \cite{R3}.

The series \eqref{nnyy} attached to the ray of slope 1
is not always
a rational functional in the variables $t,x,y$.
However, since
$$S_r = \sum_{k=0}^\infty \frac{1}{(r-1)k+1} \binom{rk}{k} t^{2k}x^ky^k$$
satisfies the polynomial equation
$$t^2xy  (S_r)^r -S_r +1 = 0,$$
the function  \eqref{nnyy} is algebraic over $\mathbb{Q}(t,x,y)$.
Whether the
functions attached to other slopes are 
algebraic over $\mathbb{Q}(t,x,y)$ is an interesting 
question (asked first by Kontsevich).

\section{Moduli of quiver representations}

\subsection{Definitions}
A {\em quiver} is a directed graph. We will consider here only
the fundamental $m$-Kronecker quiver 
$Q_m$
consisting of two vertices $\{v_1,v_2\}$ 
and $m$ edges $\{e_1, \ldots, e_m\}$
with equal orientations
$$v_1 \stackrel{e_j}{\longrightarrow} v_2\ .$$
The $m$-Kronecker quiver may be depicted with $m$ arrows as:
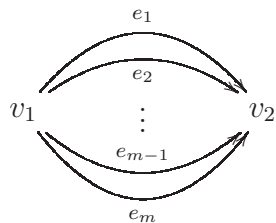
\[
\xymatrix@C=30pt
{
v_1\ar@/^30pt/[rr]^{e_1}\ar@/^20pt/[rr]_{e_2}\ar@/_23pt/[rr]^{e_{m-1}}\ar@/_33pt/[rr]_{e_m}&\vdots&v_2
}
\]

A representation of $\rho=(V_1,V_2,\tau_1, \ldots, \tau_m)$ of the quiver
$Q_m$ consists of the following linear
algebraic
data
\begin{enumerate}
\item[(i)] vector
spaces $V_i$ 
associated to the vertices $v_i$,
\item[(ii)] linear transformations $\tau_j:V_1\rightarrow V_2$ associated to
the edges $e_j$.
\end{enumerate}
While representations over any field may be studied, we 
 will restrict our attention to finite dimensional representations over
$\mathbb{C}$. Associated to $\rho$ is the {\em dimension vector}
$$\text{dim}(\rho) = (\text{dim}(V_1), \text{dim}(V_2)) \in \mathbb{Z}^2\ .$$

A {\em morphism} $\phi=(\phi_1, \phi_2)$ 
between two representations $\rho$ and $\rho'$ 
of $Q_m$ is a pair of linear tranformations
$$\phi_i: V_i \rightarrow V_i'$$
satisfying $\tau_j' \circ \phi_1 = \phi_2 \circ \tau_j$
for all $j$.
Two representations are
{\em isomorphic} if there exists a morphism $\phi$ for
which both $\phi_1$ and $\phi_2$ are
isomorphisms of vector spaces.
The notions of
sub and quotient representations are well-defined. In fact, 
the representations of $Q_m$ are easily seen to form an
abelian category. 

There are several accessible references for quiver representations.
We refer the reader to papers by King \cite{King} and Reineke \cite{R1}
where the representation theory of arbitrary quivers is treated.

\subsection{Moduli} 
Consider the
moduli space of representations of $Q_m$ with fixed dimension
vector $(d_1,d_2)$.
Let $\Hom(\mathbb{C}^{d_1}, \mathbb{C}^{d_2})$ be the space of
$d_1\times d_2$ matrices.
Every element of 
\begin{equation}\label{k23}
\shP_m(d_1,d_2)=\bigoplus_{j=1}^m \Hom(\mathbb{C}^{d_1}, \mathbb{C}^{d_2})
\end{equation}
determines a representation of $Q_m$ with dimension vector $(d_1,d_2)$.
Moreover,
the isomorphism class of every representation of $Q_m$
with dimension vector $(d_1,d_2)$
is achieved in the parameter space $\shP_m(d_1,d_2)$.

Since  $\Hom(\mathbb{C}^{d_1}, \mathbb{C}^{d_2})$ carries
canonical commuting actions of
$\mathbf{GL}_{d_1}$ and $\mathbf{GL}_{d_2}$, we obtain an action of
the product
 $\mathbf{GL}_{d_1} \times \mathbf{GL}_{d_2}$ on the
parameter space $\shP_m(d_1,d_2)$.
In fact, the scalars
$$\mathbb{C}^* \subset
\mathbf{GL}_{d_1} \times \mathbf{GL}_{d_2},$$
included diagonally
$\xi \mapsto (\xi,\xi)$
are easily seen to act trivially.
Hence, we actually have an
action of 
$$
\mathbf{G}_{d_1,d_2}= \Big( {\mathbf{GL}_{d_1} \times \mathbf{GL}_{d_2}}
\Big) \Big/ 
\mathbb{C}^*.$$

To construct an algebraic moduli space of
representations of $Q_m$, we remove the redundancy in the 
parameter space \eqref{k23} by taking
 the algebraic quotient
\begin{equation}
\label{ggt}
\shP_m(d_1,d_2)    
\Big/ \mathbf{G}_{d_1,d_2} \ .
\end{equation}
While the quotient 
  \eqref{ggt} 
is well-defined{\footnote{Quotients of 
 reductive groups actions on affine varieties can
always be taken.}}, an elementary analysis shows that there
are no nontrivial invariants \cite{R1}. Hence,
\begin{equation}
\label{ggtt}
\shP_m(d_1,d_2) 
\Big/ \mathbf{G}_{d_1,d_2} \ = \ \text{Spec}(\mathbb{C}) \ .
\end{equation}

\subsection{Stability conditions}
 The trivial quotient \eqref{ggtt} is hardly
a satisfactory answer.
Representations of $Q_m$
with dimension vector $(d_1,d_2)$ should vary in a
\begin{equation}
\label{nnt}
\text{dim}\ \shP_m(d_1,d_2) - \text{dim}\ \mathbf{G}_{d_1,d_2} =
md_1d_2 - d_1^2-d_2^2 +1\ 
\end{equation}
dimensional family.
A much richer view of the moduli of quiver
representations is obtained by imposing
stability conditions.

A {\em stability condition} $\omega$  on $Q_m$ is given by a pair of
integers $(w_1,w_2)$.
With respect to $\omega$, the slope of a representation $\rho$
of $Q_m$ with dimension vector $(d_1,d_2)$ is
$$\mu(\rho) = \frac{w_1d_1 + w_2 d_2}{d_1+d_2}\ .$$
A representation $\rho$ is {\em (semi)stable} if, for every 
proper{\footnote{Both
0 and the entire representation are excluded.}} 
subrepresentation $\widehat{\rho}\subset \rho$, 
$$\mu(\widehat{\rho})\ \ (\leq)\  < \ \ \mu(\rho)\ .$$

A central result of \cite{King} is the construction of moduli spaces
of semistable representations of quivers. Applied to $Q_m$, we
obtain the moduli space
$\mathcal{M}^{\omega}_m(d_1,d_2)$
 of $\omega$-semistable representations with dimension
vector $(d_1,d_2)$. We present here a variation  of the method of \cite{King}.

The two determinants yield two basic characters of the
group $\mathbf{GL}_{d_1}\times \mathbf{GL}_{d_2}$,
$$\text{det}_1( g_1,g_2) = \text{det}(g_1), 
\ \ \ \text{det}_2(g_1,g_2)=\text{det}(g_2)\ .$$
The stability condition $\omega$ defines a character 
$$\lambda(g_1,g_2) = \text{det}_1^{(w_2-w_1)d_2} \cdot
\text{det}_2^{(w_1-w_2)d_1} \ .$$
Since $\lambda$ is trivial on 
$\mathbb{C}^* \subset
\mathbf{GL}_{d_1} \times \mathbf{GL}_{d_2}$,
$\lambda$ descends to a character of $\mathbf{G}_{d_1,d_2}$.
Let
\begin{equation}\label{fcc}
\shP^\omega_m(d_1,d_2) = \lambda \otimes \shP_m(d_1,d_2) \oplus \lambda
\end{equation}
be the representation of  $\mathbf{G}_{d_1,d_2}$
obtained by tensoring and adding the 1-dimensional character $\lambda$
to the parameter space \eqref{k23}.
Let
$$\mathbb{P}\Big( \shP^\omega_m(d_1,d_2)\Big)^{ss}\subset
\mathbb{P}\Big( \shP^\omega_m(d_1,d_2)\Big)$$
denote the semistable locus of the canonically linearized
$\mathbf{G}_{d_1,d_2}$-action.

We are not interested in the entire variety 
$\mathbb{P}\Big( \shP^\omega_m(d_1,d_2)\Big)$.
There is a canonical
open embedding
of the parameter space \eqref{k23},
$$\shP_m(d_1,d_2) \subset \mathbb{P}\Big( \shP^\omega_m(d_1,d_2)\Big),$$
as a $\mathbf{G}_{d_1,d_2}$-equivariant open set
defined by the sum structure \eqref{fcc}.
The moduli space of $\omega$-semistable representations of
$Q_m$ with dimension vector $(d_1,d_2)$ is the
quotient
$$\mathcal{M}^{\omega}_m(d_1,d_2) = \left(
\shP_m(d_1,d_2) \ \cap \  \mathbb{P}\Big( \shP^\omega_m(d_1,d_2)\Big)^{ss}
\right) \Big/ \mathbf{G}_{d_1,d_2} \ . $$

Several important properties of the moduli space of $\omega$-semistable
representations can be deduced
from the construction \cite{King}:
\begin{enumerate}
\item[(i)]
$\mathcal{M}^{\omega}_m(d_1,d_2)$ is a  projective variety.

\item[(ii)]
An open set $\mathcal{M}^{\omega}_m(d_1,d_2)^{stable}
\subset \mathcal{M}^\omega(d_1,d_2)$  parameterizes
isomorphism classes of $\omega$-stable representations of
$Q_m$. Moreover, 
$\mathcal{M}^{\omega}_m(d_1,d_2)^{stable}$ is nonsingular
of dimension \eqref{nnt}.
\item[(iii)]
$\mathcal{M}^{\omega}_m(d_1,d_2)$ parameterizes
isomorphism classes of $\omega$-semistable representations of
$Q_m$ modulo Jordan-Holder equivalence (often called $S$-equival\-ence).

\end{enumerate}
While properties (ii) and (iii) hold for stability conditions
on arbitrary quivers, property (i) is special to $Q_m$.
By the results of \cite{King}, $\mathcal{M}^{\omega}_m(d_1,d_2)$
is projective over the quotient \eqref{ggtt}. Since
the quotient \eqref{ggtt} is $\text{Spec}(\mathbb{C})$, 
the moduli space $\mathcal{M}^{\omega}_m(d_1,d_2)$
is a projective variety.

If $\omega=(0,0)$, all representations are semistable. Then,
$$\mathcal{M}^{(0,0)}_m(d_1,d_2)= 
\shP_m(d_1,d_2)   
\Big/ \mathbf{G}_{d_1,d_2} \ = \ \text{Spec}(\mathbb{C}) \ $$
as before.
By the following result of Reineke \cite{R1},
we will restrict our attention to the stability conditions
$(1,0)$ and $(0,1)$.

\begin{lemma} 
$\omega$-(semi)stability is equivalent to (semi)stability 
with respect to either $(0,0)$, $(1,0)$, or $(0,1)$.
\end{lemma}

\begin{proof}
Let $\omega=(w_1,w_2)$.
By the definition of (semi)stability of representations,
we see
$\omega$-(semi)stability is equivalent to both
\begin{enumerate}
\item[(i)] $(w_1+\gamma,w_2+\gamma)$-(semi)stability for
           $\gamma\in \mathbb{Z}$ and
\item[(ii)] $(\lambda w_1, \lambda w_2)$-(semi)stability for
            $\lambda\in \mathbb{Z}>0$.
\end{enumerate}  
If $w_1=w_2$, then $\omega$-(semi)stability is equivalent to
$(0,0)$-(semi)stability by (i).
If $w_1>w_2$, then $\omega$-(semi)stability is equivalent to
$(w_1-w_2,0)$-(semi)stability by (i) and then 
$(1,0)$-(semi)stability by (ii). Similarly, the $w_1<w_2$
case leads to 
$(0,1)$-(semi)stability.
\end{proof}

\subsection{Framing} \label{fr11}
Strictly semistable representations of $Q_m$ usually lead to singularities
of the moduli space $\mathcal{M}_m^\omega(d_1,d_2)$.
Following \cite{ER}, we introduce framing data to improve the
moduli behaviour.

We consider two types of framings for representations of $Q_m$.
A {\em back} framed representation of $Q_m$ is a
pair 
$(\rho, L_1)$
where $\rho=(V_1,V_2,\tau_1,\ldots, \tau_m)$ is  standard
 representation of $Q_m$ and
$L_1\subset V_1$ is a 1-dimensional subspace.
A {\em front} framed representation of $Q_m$ is a
pair 
$(\rho, L_2)$
where $L_2\subset V_2$ is a 1-dimensional subspace.
The subspaces $L_i$ are the framings.
Two framed representations are isomorphic if the
underlying standard representations admit an
isomorphism preserving the framing.

A stability condition $\omega$ for $Q_m$ induces
a canonical notion of stability for framed representations.
A framed representation $(\rho,L_i)$
is {\em stable} if the following two conditions hold:
\begin{enumerate}
\item[(i)] $\rho$ is an $\omega$-semistable representation,
\item[(ii)]  for every proper subrepresentation $\widehat{\rho} \subset
\rho$ containing $L_i$,
$$\mu(\widehat{\rho}) < \mu(\rho).$$
\end{enumerate}
The moduli of stable framed representations admits a
 GIT quotient construction with no strictly
semistables. In fact, stable framed representations
can be viewed as stable standard representations for
quivers obtained by augmenting $Q_m$ by one vertex
(and considering appropriate standard stability
conditions). We refer the reader to \cite{ER} for a detailed
discussion.

Let $\shM_m^{\omega,B}(d_1,d_2)$ and $\shM_m^{\omega,F}(d_1,d_2)$
denote the moduli spaces of back and front framed
representations of $Q_m$. Both are nonsingular, irreducible,
projective varieties.
 
\subsection{Examples: stability condition $(0,1)$}
Consider first the stability condition $(0,1)$
on the quiver $Q_m$. Suppose $\rho$ is a
standard representation with dimension
vector $(d_1,d_2)$ satisfying $d_1,d_2>0$. 
There exists a proper subrepresentation
$$\widehat{\rho}= (0,\widehat{V}_2, 0, \ldots, 0)$$
where $\widehat{V}_2\subset V_2$ is any 1 dimensional
subspace. We see
$$\mu(\widehat{\rho}) = \frac{1}{1}> \frac{d_2}{d_1+d_2}=\mu(\rho) \ .$$
Hence, $\rho$ can not be $(0,1)$-semistable.

The dimension vectors of $(0,1)$-semistable representations
of $Q_m$ must be parallel to either $(1,0)$ or $(0,1)$.
In fact, if framings are placed, only the dimension vectors
$(1,0)$ and $(0,1)$ are possible. Elementary considerations
yield the following result.

\begin{lemma} The moduli space of stable framed representations 
of $Q_m$ with respect to the condition $(0,1)$ is a point
in the two cases
$$\shM_{m}^{(0,1),B}(1,0), \ \  \ \ \shM_{m}^{(0,1),F}(0,1),$$
and empty otherwise.
\end{lemma}

\subsection{Examples: stability condition $(1,0)$}
The stability condition $(1,0)$
on the quiver $Q_m$ leads to much more interesting behavior.
Unlike the $(0,1)$ condition, we will here be only able 
to undertake a case by case analysis.

For the $1$-Kronecker quiver $Q_1$, the moduli spaces
of stable framed representations must have dimension
vectors equal to $(1,0)$, $(0,1)$, or $(1,1)$. Again,
in all four cases (for possible back and front framing), the
moduli spaces are points.

For the $2$-Kronecker quiver, we find a richer set of possibilities
of $(1,0)$-semistable representations.

\begin{lemma}\label{hnt} 
If $\rho$ is a $(1,0)$-semistable representation of $Q_2$,
then the dimension vector must be proportional to one of  
$$(k,k+1), \ \ (1,1), \ \ (k+1,k) $$
for $k\geq 1$.
\end{lemma}

\begin{proof} Suppose $\rho=(V_1,V_2,\tau_1,\tau_2)$ is a
representation of $Q_2$. 
We analyze first the case where $d_1 < d_2$.
The case $d_1> d_2$ is obtained by dualizing.{\footnote{The dual of
$\rho$ is $\rho^*=(V_2^*,V_1^*, \tau_1^*, \tau_2^*)$, and
$\rho$ is $(1,0)$-semistable if and only if $\rho^*$ is
$(1,0)$-semistable.}}

Since the slope of $\rho$ is $\frac{d_1}{d_1+d_2}$,  
$(1,0)$-semistabiliy is violated if there exists a
non-trivial subspace $\widehat{V}_1\subset V_1$ satisfying
\begin{equation}\label{btty}
\frac{\text{dim}(\widehat{V}_1)}{\text{dim}(\widehat{V}_1) +
\text{dim}\left(\tau_1(\widehat{V}_1)+\tau_2(\widehat{V}_1)\right)} >
\frac{d_1}{d_1+d_2} \ .
\end{equation}
If $\rho$ is $(1,0)$-semistable, the maps $\tau_1$
and $\tau_2$ must be injective (by taking
$\widehat{V_1}$ to be $\text{Ker}(\tau_i)$).

We now assume $\rho$ to be $(1,0)$-semistable and
construct a candidate for $\widehat{V}_1$ by the following
method. Let $S_0= V_1$, and let
$$S_{i} = \tau_1^{-1}\left(\tau_2(S_{i-1})\right)
\ \ \ \text{for}\ \ i>0 \ .$$ Since $S_i \subset S_{i-1}$, we obtain a
filtration
$$\ldots \subset S_3 \subset S_2 \subset S_1 \subset S_0 \ . $$
If $S_i$ is nonempty, then the inclusion  $S_i \subset S_{i-1}$
must be proper (otherwise $\widehat{V}_1=S_i$ violates
\eqref{btty}).
Since the codimension of $S_i \subset V_1$ is at most
$i(d_2-d_1)$, we see
$$S_{\lfloor \frac{d_1-1}{d_2-d_1} \rfloor} \neq 0 \ .$$
We can find a sequence of elements 
$\epsilon_i \in S_i \setminus S_{i+1}$ for 
$0\leq i \leq \lfloor \frac{d_1-1}{d_2-d_1} \rfloor$
such that
$$\tau_2(\epsilon_i) = \tau_1(\epsilon_{i+1})\ .$$
Let $\widehat{V}_1$ be span of $\epsilon_0, \ldots, 
\epsilon_{\lfloor \frac{d_1-1}{d_2-d_1} \rfloor}$.

Since the $\epsilon_i$ are independent, the dimension of
$\widehat{V}_1$ is $\lfloor \frac{d_1-1}{d_2-d_1} \rfloor +1$.
The dimension of
$\tau_1(\widehat{V}_1)+\tau_2(\widehat{V}_1)$ is at most
$\lfloor \frac{d_1-1}{d_2-d_1} \rfloor +2$, so
\begin{equation*}
\frac{\text{dim}(\widehat{V}_1)}{\text{dim}(\widehat{V}_1) +
\text{dim}\left(\tau_1(\widehat{V}_1)+\tau_2(\widehat{V}_1)\right)} \geq
\frac{\lfloor \frac{d_1-1}{d_2-d_1} \rfloor +1}
{2\lfloor \frac{d_1-1}{d_2-d_1} \rfloor +3}
\end{equation*}
Therefore, since $\rho$ is $(1,0)$-semistable,
we must have
$$\frac{\lfloor \frac{d_1-1}{d_2-d_1} \rfloor +1}
{2\lfloor \frac{d_1-1}{d_2-d_1} \rfloor +3} \leq \frac{d_1}{d_1+d_2}$$
or, equivalently,
\begin{equation}\label{kw2}
(d_2-d_1){\lfloor \frac{d_1-1}{d_2-d_1} \rfloor +d_1+d_2}\leq 3d_1\ .
\end{equation}

There are now two cases. If $d_2-d_1$ divides $d_1-1$, then
the inequality immediately implies $d_2=d_1+1$. 
If $d_2-d_1$ does not divide $d_1-1$, the inequality implies
$d_2-d_1$ divides $d_1$. In the second case, the dimension
vector is proportional to $(\frac{d_1}{d_2-d_1}, \frac{d_1}{d_2-d_1}+1)$.
\end{proof}

The construction of $(1,0)$-semistable representations of
$Q_2$ with dimension vectors in the directions permitted
by Lemma \ref{hnt} is an easy exercise.
We will discuss in more detail the directions $(1,2)$ and $(1,1)$.

The moduli spaces of stable back framed representation of $Q_2$
of dimension vector $(k,2k)$
are
empty for $k\geq 2$ and
$\shM_2^{(1,0),B}(1,2)$ is a point. Front framing
is slightly more complicated,
$$\shM_2^{(1,0),F}(1,2)= \mathbb{P}^1, \ \ \shM_2^{(1,0),F}(2,4)=\text{point},$$
and $\shM_2^{(1,0),F}(k,2k)$ is empty for $k>2$.
These results are obtained by simply unravelling the
definitions.

For dimension vector proportional to $(1,1)$, the framed
moduli spaces are always nonempty. Their topological Euler characteristics
are determined by the following result.

\begin{lemma} For $k\geq 1$, we have 
$\chi\big(\shM_2^{(1,0),B}(k,k)\big) = \chi\big(
\shM_2^{(1,0),F}(k,k)\big)= k+1$. 
\end{lemma}

\begin{proof}
The simplest approach is to count the fixed points of  the 
$\mathbb{C}^* \times \mathbb{C}^*$-action on
the framed moduli spaces obtained by scaling $\tau_1$ and 
$\tau_2$,
$$(\xi_1,\xi_2) \cdot \left( \Big(\mathbb{C}^k,\mathbb{C}^k,\tau_1,\tau_2
\Big), L_i
 \right) =
\left(
 \Big(\mathbb{C}^k,\mathbb{C}^k,\xi_1 \tau_1,\xi_2\tau_2 \Big),L_i\right)\ .$$
Certainly, 
$\shM_2^{(1,0),B}(1,1)$
and  $
\shM_2^{(1,0),F}(1,1)$
are both $\mathbb{P}^1$ with fixed points given by
$$\tau_1=1,\  \tau_2=0, \ \ \text{and} \ \ \tau_1=0, \ \tau_2=1\ $$
and unique choice for the framings.

The moduli spaces with dimension vector $(2,2)$ are the
first nontrivial cases.
Two $2\times 2$ matrices together with a non-zero vector
in $\mathbb{C}^2$  
specify a back framed representation of $Q_2$.
The three $\mathbb{C}^* \times \mathbb{C}^*$-fixed
points of $\shM_2^{(1,0),B}(2,2)$ are given by
the data
$$\left\{ \  \tau_1=\left(\begin{array}{cc} 1& 0 \\ 0 & 1 \end{array}\right),\ 
\tau_2=\left(\begin{array}{cc} 0& 1 \\ 0 & 0 \end{array}\right),\ 
L_1=\left(\begin{array}{c} 0 \\ 1  \end{array}\right) \right\},$$

$$\left\{ \  \tau_1=\left(\begin{array}{cc} 0& 1 \\ 0 & 0 \end{array}\right),\ 
\tau_2=\left(\begin{array}{cc} 1& 0 \\ 0 & 1 \end{array}\right),\ 
L_1=\left(\begin{array}{c} 0 \\ 1  \end{array}\right) \right\},$$

$$\left\{ \  \tau_1=\left(\begin{array}{cc} 1& 0 \\ 0 & 0 \end{array}\right),\ 
\tau_2=\left(\begin{array}{cc} 0& 0 \\ 0 & 1 \end{array}\right),\ 
L_1=\left(\begin{array}{c} 1 \\ 1  \end{array}\right) \right\}.$$
The analysis for $\shM_2^{(1,0),F}(2,2)$ is similar.
We leave the higher $k$ examples for the reader to investigate.

A treatment of torus actions on moduli of spaces of representations
of quivers can be found in \cite{TTT}. In fact,
$\shM_2^{(1,0),B}(k,k) \stackrel{\sim}{=}
\shM_2^{(1,0),F}(k,k) \stackrel{\sim}{=} \mathbb{P}^k$. 
\end{proof}

\subsection{Reineke's Theorem}\label{fr22}
The main result relating commutators in the tropical
vertex group to the Euler characteristics of the moduli spaces of 
representations of $Q_m$ can now be stated.
Consider the elements
$$S_{m}=\theta_{(1,0),(1+tx)^{m}} \ \  \text{and} \ \ 
T_{m}=\theta_{(0,1), (1+ty)^{m}}\ $$
of the tropical vertex group.
The unique
factorization
\begin{equation}\label{jjtrr}
    T_{m}^{-1} \circ S_{m} \circ T_{m}  \circ S_{m}^{-1} =
\stackrel{\rightarrow} \prod \theta_{(a,b),f_{a,b}}\ 
\end{equation}
associates a function
$$f_{a,b}\in \mathbb{C}[x^ay^b][[t]]$$
to every   primitive
vector $(a,b)\in \mathbb{Z}^2$ lying strictly in the first quadrant.
Two more functions are obtained from the
 topological Euler characteristics of the
moduli spaces of back and front framed representations of $Q_m$,
$$B_{a,b} = 1+\sum_{k\geq 1} \chi\Big( \shM^{(1,0),B}_m(ak,bk)\Big)\cdot  
(tx)^{ak}\ (ty)^{bk}\ ,$$
$$F_{a,b} = 1+\sum_{k\geq 1} \chi\Big( \shM^{(1,0),F}_m(ak,bk)\Big)\cdot  
(tx)^{ak}\ (ty)^{bk} 
\ .$$

\vspace{10pt}
\noindent{\bf Theorem 1.} (Reineke) {\em The three
functions are related by the equations}
$$f_{a,b}=(B_{a,b})^{\frac{m}{a}}= (F_{a,b})^{\frac{m}{b}}\ . $$
\vspace{0pt}

Theorem 1 is proven in \cite{R2}. Reineke calculates the Euler
characteristics of the framed moduli spaces by counting points
over finite fields. The connection to the
tropical vertex group is made via a homomorphism from the
Hall algebra following the wall-crossing philosophy of
\cite{ks2}. The relevant wall-crossing is from the $(0,1)$
to $(1,0)$ stability condition.
The ordered product factorization is then
obtained from the Harder-Narasimhan filtration in the
abelian category of representations of $Q_m$.

\subsection{Examples}
For $Q_1$, 
 the moduli spaces 
of framed representations are empty for
slopes (strictly in the first quadrant) other than 1.
Moreover, 
$\shM^{(1,0),B}_1(k,k)$ and $\shM^{(1,0),F}_1(k,k)$
are points if $k=1$ and empty otherwise.
Theorem 1 then 
immediately recovers the commutator calculation of
Figure
\ref{diagram11}.

 For $Q_2$ and primitive vector $(a,b)=(1,2)$,
the results of Section 2.6 yield
\begin{eqnarray*}
B_{1,2} & =& 1+ t^3 x y^2\ ,\\
F_{1,2} & =& 1+2t^3 xy^2+ t^6 x^2y^4\ .
\end{eqnarray*}
By the commutator results of Section \ref{exxc}, we see
$$f_{1,2}= (1+t^3 xy^2)^2$$
verifying Theorem 1.
For $Q_2$ and primitive vector $(a,b)=(1,1)$,
we obtain 
\begin{eqnarray*}
B_{1,1} & =& (1-  t^2x y)^{-2}\ ,\\
F_{1,1} & =& (1-t^2 xy)^{-2}\ .
\end{eqnarray*}
By the commutator results of Section \ref{exxc}, we see
$$f_{1,1}= (1+t^2 xy)^{-4}$$
again verifying Theorem 1.

\section{Rational curves on toric surfaces}

\subsection{Toric surfaces} \label{ttss}

Let $(a,b)\in \mathbb{Z}^2$ be a primitive vector lying 
strictly in the first quadrant.
The rays generated by $(-1,0)$, $(0,-1)$, and $(a,b)$ determine
a complete rational fan{\footnote{We refer the reader
to \cite{ful} for background on toric varieties.}} in $\mathbb{R}^2$,
see Figure \ref{P2abfan}.

\begin{figure}
\input{P2abfan.pstex_t}
\caption{}
\label{P2abfan} 
\end{figure}

Let $X_{a,b}$
be the associated toric surface
 with toric divisors
$$D_1, D_2, D_\out \subset X_{a,b}$$
corresponding to the respective rays.
Concretely,
$X_{a,b}$ is the weighted projective plane obtained
by the quotient
$$X_{a,b} = \big( \mathbb{C}^3 -  \{0\}\big) \big/ \mathbb{C}^*$$
where the $\mathbb{C}^*$-action is given by
$$\xi\cdot(z_1,z_2,z_3) = (\xi^a z_1,\xi^bz_2,\xi z_3) \ .$$
The divisors $D_1$, $D_2$ and $D_\out$
 correspond respectively  to the vanishing loci of $z_1$, $z_2$, and $z_3$.

Let 
$X_{a,b}^o \subset X_{a,b}$
be the open surface obtained by removing the three
toric fixed points
$$[1,0,0], \ [0,1,0], \ [0,0,1]\ .$$
 Let 
$D_1^o, D_2^o, D_\out^o$ be the restrictions of the
toric divisors to $X_{a,b}^o$.

We denote {\em ordered 
partitions} ${\bf Q}$ of length $\ell$ by  $q_1+ \ldots + q_{\ell}$. 
Ordered partitions differ from usual partitions in two basic ways.
First,
the ordering of the parts matters.
Second, the parts $q_i$ are
required only to be non-negative integers (0 is permitted).
The {size}  $|{\bf Q}|$ is the sum of the parts.

Let $k\geq 1$.
Let 
${\bf P}_a=p_1+ \ldots+ p_{\ell_1}$ and ${\bf P}_b=p'_1 + \ldots +p'_{\ell_2}$ 
be ordered partitions of size $a k$ 
and $bk$ respectively. Denote the pair by ${\bf{P}}=({\bf P}_a,{\bf P}_b)$.
Let
$$\nu: X_{a,b}[{\bf P}] \rightarrow X_{a,b}$$
be the blow-up of $X_{a,b}$ along $\ell_1$ and $\ell_2$
distinct points of $D^o_1$ and $D^o_2$.
Let 
$$X^o_{a,b}[{\bf P}] = \nu^{-1}(X^o_{a,b}).$$

Let $\beta_k \in H_2(X_{a,b},\mathbb{Z})$ be the unique
class with intersection numbers
$$\beta_k \cdot D_1 =ak , \ \ \beta_k \cdot D_2 = bk, \ \ 
\beta_k\cdot D_\out =k.$$
Let $E_i$ and $E'_j$ be the $i^{th}$ and $j^{th}$
exceptional divisors over $D^o_1$ and $D^o_2$.
Let $$\beta_k[{\bf P}] = \nu^*(\beta_k) -\sum_{i=1}^{\ell_1} p_i [E_i]
-\sum_{j=1}^{\ell_2} p'_j [E'_j]\ \in H_2(X_{a,b}[{\bf P}], \mathbb{Z}) .$$

\label{ggb}

\subsection{Moduli of maps} \label{mom}
Let $\overline{\foM}(X^o_{a,b}[{\bf P}]/D_\out^o)$ denote the moduli 
space of stable relative maps\footnote{We refer the
reader to \cite{junex} for an introduction to
relative stable maps.} of genus 0 curves representing the
class $\beta_k[{\bf P}] $ and with full contact order  
$k$ at an unspecified point of $D^o_\out$.
By Proposition 4.2 of \cite{GPS}, the moduli space
 $\overline{\foM}(X^o_{a,b}[{\bf P}]/D_\out^o)$ is
proper (even though the target geometry is open).
We can easily calculate the virtual dimension,
\begin{eqnarray*}
\text{dim}^{vir} \ \overline{\foM}(X^o_{a,b}[{\bf P}]/D_\out^o) & = &
c_1(X^o_{a,b}[{\bf P}])\cdot \beta_k[{\bf P}] - 1 - (k-1) \\ & = &
\left(\nu^*c_1(X^o_{a,b})
-\sum_{i=1}^{\ell_1} [E_i]
-\sum_{j=1}^{\ell_2} [E'_j]
\right)\cdot \beta_k[{\bf P}] - k \\
& = & ak+bk +k -ak-bk -k \\
& = & 0\ ,
\end{eqnarray*}
where the formula for the Chern class of a toric variety,
$$c_1(X^o_{a,b})=D_1+D_2+D_\out,$$ is used
in the second line.

Since $\overline{\foM}(X^o_{a,b}[{\bf P}]/D_\out^o)$ is proper
of virtual dimension 0, we may define the associated
Gromov-Witten invariant by
$$N_{a,b}[{\bf P}] 
= \int_{[\overline{\foM}(X^o_{a,b}[{\bf P}]/D_\out^o)]^{vir}}1 \ \in \mathbb{Q} \ .$$
Proposition 4.2 of \cite{GPS} shows  
$N_{a,b}[{\bf P}]$ does {\em not} depend upon the
locations of the blow-ups of $X^0_{a,b}$.

Naively, $N_{a,b}[{\bf P}]$ counts rational curves on $X_{a,b}^0$
with full contact at a single (unspecified) point of $D_\out$ and
 with specified multiple points of orders given
by $\bf{P}$ on $D_1^0$ and $D_2^0$. However, the moduli space
$\overline{\foM}(X^o_{a,b}[{\bf P}]/D_\out^o)$
may include multiple covers and components of excess dimension. 
In particular, $N_{a,b}[{\bf P}]$ need not be integral (nor even
positive).

\subsection{Formula}
The main result relating commutators in the tropical vertex
group to rational curve counts on toric surfaces can now be stated.
Consider the elements
$$S_{\ell_1}=\theta_{(1,0),(1+tx)^{\ell_1}} \ \  \text{and} \ \ 
T_{\ell_2}=\theta_{(0,1), (1+ty)^{\ell_2}}\ $$
of the tropical vertex group.
The unique
factorization
\begin{equation}\label{jjtrrr}
    T_{\ell_2}^{-1} \circ S_{\ell_1} \circ T_{\ell_2}  
\circ S_{\ell_1}^{-1} =
\stackrel{\rightarrow} \prod \theta_{(a,b),f_{a,b}}\ 
\end{equation}
associates a function
$$f_{a,b}\in \mathbb{C}[x^ay^b][[t]]$$
to every   primitive
vector $(a,b)\in \mathbb{Z}^2$ lying strictly in the first quadrant.
Since the series $f_{a,b}$ 
 starts with 1, we may take the logarithm.
Homogeneity constraints determine the behavior of the
variable $t$. We define the coefficients
$c^k_{a,b}(\ell_1,\ell_2) \in \mathbb{Q}$ by
$$\log f_{(a,b)} = \sum_{k\geq 1}  k\ c^k_{a,b}(\ell_1,\ell_2) \cdot  
(tx)^{ak}\ (ty)^{bk} .$$
The function $f_{a,b}$ is linked to Gromov-Witten theory
by the following result proven in \cite{GPS}.

\vspace{10pt}
\noindent{\bf Theorem 2.} {\em
We have
$$c^k_{a,b}(\ell_1,\ell_2) = \sum_{|{\bf P}_a| = ak} \ \sum_{|{\bf P}_b|=bk} 
N_{a,b}[({\bf P}_a,{\bf P}_b)]$$
where the sums are over all ordered
partitions ${\bf P}_a$ of size {\em ak} and length $\ell_1$
  and ${\bf P}_b$
of size {\em bk} and length $\ell_2$.} 
\vspace{10pt}

The proof of Theorem 2 starts with the relationship of
the tropical vertex group to tropical curve counts
on toric surfaces. A transition to holomorphic curve
counts with relative constraints is made via \cite{nisi}. Finally,
a degeneration argument is used to separate the virtual and
enumerative geometry of the invariant $N_{a,b}[{\bf P}]$. 
The virtual aspects are handled by the multiple cover
formulas of \cite{Brypan} and the enumerative aspects
by the tropical/holomorphic curve counts.

\subsection{Examples} We consider the examples of \S\ref{exxc},
focusing on the functions attached to the ray of slope $1$. For
$\ell_1=\ell_2=1$, 
\[
\log f_{1,1}=\log (1+t^2xy)=\sum_{k=1}^{\infty} k \cdot {(-1)^{k+1}
\over k^2} \cdot (tx)^k (ty)^k.
\]
Consider $\PP^2$ with the three
toric divisors $D_1$, $D_2$ and $D_{\out}$ making up the toric boundary. 
There is a unique line passing through a point selected
on $D_1$ and a  point selected on 
$D_2$. Hence,
$N_{1,1}[(1,1)]=1$. There are no
other rational curves in $\PP^2$ passing through these two points
and maximally tangent to $D_{\out}$.
The result
 $$N_{1,1}[(k,k)]=\frac{(-1)^{k+1}}{k^2}$$
comes from multiple covers of the line totally branched over the
intersection with $D_{\out}$. The multiple
cover contribution is computed in \cite{Brypan}.

Next, consider the ray of slope 1 for $\ell_1=\ell_2=2$. We calculate
\[
\log f_{1,1}=-4\log(1-t^2xy)=4\sum_{k=1}^{\infty}
k\cdot {1\over k^2}\cdot (tx)^k (ty)^k.
\]
We now must choose two points each on $D_1$ and $D_2$. As above,
$N_{(1,1)}[(1+0,1+0)]=1$ because there is exactly
one line through two points. Similarly 
$$N_{1,1}[(1+0,0+1)]=
N_{1,1}[(0+1,1+0)]=N_{1,1}[(0+1,0+1)]=1,$$
 giving the desired
total for $c^1_{1,1}(2,2)=4$. 
The invariant
 $$N_{1,1}[(2+0,2+0)]=-1/4$$
is obtained 
from the double covers of the line.
Hence, double covers of the four lines contribute $-1$ to $c^2_{1,1}(2,2)$.
On the other hand, there is a pencil of conics passing through the
four chosen points. Being tangent to $D_{\out}$ is a quadratic condition,
so $$N_{1,1}[(1+1,1+1)]=2.$$ 
Putting the calculation together yields
$$c^2_{1,1}(2,2)=(-1)+2=1.$$
 All remaining contributions
to $c^k_{1,1}(2,2)$ for $k>2$ come from multiple covers of either
one of the lines or one of the conics.

For the ray of slope 1 for $\ell_1=2$, $\ell_2=3$, we have
\[
\log f_{1,1}=6 (tx)(ty)+2\cdot {9\over 2} (tx)^2(ty)^2
+3\cdot {20\over 3} (tx)^3(ty)^3+\cdots.
\]
The coefficient $c^{1}_{1,1}(2,3)=6$ counts the number of lines passing
through one of two points on $D_1$ and one of three points on $D_2$.
The coefficient 
$$c^{2}_{1,1}(2,3)=9/2=6-6/4$$
is obtained as follows.
There are six conics
passing through the two chosen points on $D_1$ and two of the three
chosen points on $D_2$ and tangent to $D_{\out}$. The $-6/4$ accounts
for double covers of the lines. It is possible to compute 
$$N_{1,1}[2+1,1+1+1]=N_{1,1}[1+2,1+1+1]=3.$$
These are the only
contributions from non-multiple covers to $c^3_{1,1}(2,3)$ ---
corresponding to plane cubics with a node at one of the two chosen
points on $D_1$ and passing through all chosen points, with $D_{\out}$
being an inflectional tangent. On the other
hand, the triple covers of each line contribute $1/9$, for a total
of 
$$c^3_{1,1}(2,3)=3+3+6/9=20/3.$$
 For higher $k$, there continue
to be contributions from curves which are not just multiple covers
of curves already found.

\subsection{Correspondence}
Theorems 1 and 2 together yield an interesting correspondence
between the moduli space of rational curves on toric sufaces and
the moduli spaces of quiver representations.

\vspace{10pt}
\noindent{\bf Corollary 3.} {\em For every $m>0$ and primitive
$(a,b) \in \mathbb{Z}^2$ lying strictly in the first quadrant, 
we have
\begin{multline*}
\exp\left( \sum_{k\geq 1}  \sum_{|{\bf P}_a| = ak} \ \sum_{|{\bf P}_b|=bk} 
k N_{a,b}[({\bf P}_a,{\bf P}_b)]
 \cdot  
(tx)^{ak}\ (ty)^{bk} \right)\\ =
\left( 1+\sum_{k\geq 1} \chi\Big( \shM^{(1,0),B}_m(ak,bk)\Big)\cdot  
(tx)^{ak}\ (ty)^{bk}
\right)^{\frac{m}{a}}\\
= 
\left(1+\sum_{k\geq 1} \chi\Big( \shM^{(1,0),F}_m(ak,bk)\Big)\cdot  
(tx)^{ak}\ (ty)^{bk} \right)^{\frac{m}{b}} \ 
\end{multline*} 
where the sums in the first line are over all ordered
partitions ${\bf P}_a$ of size {\em ak} and length $m$
  and ${\bf P}_b$
of size {\em bk} and length $m$.}
\vspace{10pt}

Corollary 3 is a correspondence between rational curve counts for 
the toric surface $X_{a,b}$ and Euler characteristics
of framed moduli spaces of quiver representations
of $Q_m$ with dimension vectors proportional to $(a,b)$.
At the moment, no direct geometric argument for Corollary 3
is known. 
Also, while  parallels between Corollary 3 and the correspondences
of \cite{PT} are apparent (both link Gromov-Witten invariants
to possibly virtual
Euler characteristics of moduli spaces of framed sheaves),
again no precise connection is known.

Theorem 2 as stated is more general than Theorem 1 since $\ell_1$ and
$\ell_2$ are not required to be equal. Richer versions of Theorem 1 
which capture the $\ell_1\neq \ell_2$ cases can be obtained
from more complicated quiver constructions.{\footnote{M. Reineke
has explained to us a method using certain bipartite
quivers (up to symmetric group actions). A. King has made a
similar proposal.}}
Finally, a version of Theorem 2 which casts the commutator
calculations in the tropical vertex group (over
many variables instead of just $t$) as equivalent to
the determination of the invariants 
$N_{a,b}[({\bf P}_a,{\bf P}_b)]$ can be found in \cite{GPS}.

\section{Scattering patterns}

\subsection{Directions}
Consider the basic elements
$$S_{\ell_1}=\theta_{(1,0),(1+tx)^{\ell_1}} \ \  \text{and} \ \ 
T_{\ell_2}=\theta_{(0,1), (1+ty)^{\ell_2}}\ $$
of the tropical vertex group.
The unique
factorization
\begin{equation}\label{jjtrrrr}
    T_{\ell_2}^{-1} \circ S_{\ell_1} \circ T_{\ell_2}  
\circ S_{\ell_1}^{-1} =
\stackrel{\rightarrow} \prod \theta_{(a,b),f_{a,b}}\ 
\end{equation}
associates a function
$$f_{a,b}\in \mathbb{C}[x^ay^b][[t]]$$
to every   primitive
vector $(a,b)\in \mathbb{Z}^2$ lying strictly in the first quadrant.

\vspace{10pt}
\noindent{\bf Question 4.} For which directions is $f_{a,b} \neq 1$ ?
\vspace{10pt}

The {\em scattering pattern} associated to $\ell_1$ and $\ell_2$
consists of the directions in the first quadrant for which
$f_{a,b}\neq 1$. We have seen several examples of scattering
patterns in Section \ref{exxc}. Our goal here is to give
an answer to Question 4 via Theorem 2 and the
the classical geometry of curves on toric surfaces.

\subsection{Curves}
If $f_{a,b} \neq 1$, then there must exist, by Theorem 2,
 a nonvanishing invariant
\begin{equation*}
N_{a,b}[({\bf P}_a,{\bf P}_b)] \neq 0,
\end{equation*}
where  ${\bf P}_a$ is of size {\em ak} and length $\ell_1$
  and ${\bf P}_b$
of size {\em bk} and length $\ell_2$.
The nonvanishing of the invariant implies the nonemptiness
of the corresponding moduli space,
$$\overline{\foM}(X^o_{a,b}[({\bf P}_a,{\bf P}_b)]/D_\out^o) \neq \emptyset \ . $$
Recall, following the notation of Section \ref{ggb},
$$\nu: X^o_{a,b}[({\bf P}_a,{\bf P}_b)]\rightarrow X^o_{a,b}$$
is the blow-up along $\ell_1$ and $\ell_2$ distinct points
of $D_1^o$ and $D_2^o$ respectively.

Let $[\phi]\in \overline{\foM}(X^o_{a,b}[({\bf P}_a,{\bf P}_b)]/D_\out^o)$
be a stable relative map,
$$(C,p) \stackrel{\phi}{\rightarrow} \foX^o_{a,b}[({\bf P}_a,{\bf P}_b)]
\stackrel{\pi}\rightarrow
X^o_{a,b}[({\bf P}_a,{\bf P}_b)]
,$$
satisfying the following properties:
\begin{enumerate}
\item[(i)] $C$ is a complete connected curve of arithmetic genus 0 with at worst nodal
            singularities,
\item[(ii)] $\foX^o_{a,b}[({\bf P}_a,{\bf P}_b)]\rightarrow
X^o_{a,b}[({\bf P}_a,{\bf P}_b)]$ is a destabilization{\footnote{A destabilization
along a relative divisor is obtained by attaching
a finite number of bubbles each of which is
a $\mathbb{P}^1$-bundle over the divisor.
We refer the reader to 
Section 1 of \cite{junex} for an introduction to the destabilizations
required for stable relative maps. Li uses the term
{\em expanded degeneration} for our destabilizations.}} 
along the
relative divisor $D_\out^o$,
\item[(iii)] $C$ has full contact via $\phi$ with  $D_\out^o$ of order $k$ at $p$.
\end{enumerate}
For the calculation of intersection numbers, we will often view the
composition 
$$\pi \circ \phi: C \rightarrow X^o_{a,b}[({\bf P}_a,{\bf P}_b)]\subset
X_{a,b}[({\bf P}_a,{\bf P}_b)]$$
as having image in the complete surface.
Let 
$$D^{\mathrm{strict}}_i \subset X_{a,b}[({\bf P}_a,{\bf P}_b)]$$
be the strict transformation under $\nu$ of $D_i$.

\begin{lemma} Let $C'\subset C$ be an irreducible component on
which $\pi \circ \phi$ is nonconstant. Then,
\begin{equation*}
C' \cdot  D^{\text{\em strict}}_1 =  C' \cdot D^{\text{\em strict}}_2 = 0 \ .
\end{equation*} \label{jj23}
\end{lemma}
\begin{proof}
Since $\pi\circ \phi(C')\subset X^o_{a,b}[({\bf P}_a,{\bf P}_b)]$,
the component $C'$ 
can not dominate $D^{\text{strict}}_i$. Hence,
$$
C' \cdot  D^{\text{strict}}_i \geq 0 \ .$$
The intersection number of
$C$ with  $D^{\text{strict}}_1$ is
$$ C\cdot D^{\text{strict}}_1 = \beta_k \cdot D_1 + \sum_{i=1}^{\ell_1} p_i E_i^2 = 0$$
where ${\bf P}_a= p_1+ \ldots+ p_{\ell_1}$ and
$E_i$ are the exceptional divisors of $\nu$ over $D_1$.
Therefore, if $C'\cdot D^{\text{strict}}_1>0$, then
\begin{equation*}
\overline{C\setminus C'} \cdot D^{strict}_1<0 \ 
\end{equation*}
which is impossible since no component of
$C$ dominates $D^{\text{strict}}_1$.
The argument for $D^{\text{strict}}_2$ is identical.
\end{proof}

\begin{lemma} Let $C'\subset C$ be an irreducible component on
which $\pi \circ \phi$ is nonconstant.
The set \label{jj24}
$$C' \ \cap \ (\pi \circ \phi)^{-1}(D_\out^o)$$ 
consists of a single point.
\end{lemma}
\begin{proof}
Let $q=\pi\circ \phi(p)\in D_\out^o$. 
Since no components of $C$ dominate $D_\out$ and $\phi(C)$ has full contact
with the extremal $D_\out \subset \foX^o_{a,b}[({\bf P}_a,{\bf P}_b)]$
at a single point, we conclude
$\pi \circ \phi(C')$ meets $D_\out^o$ only at $q$.
Since the dual graph of $C$ has no loops (by the genus 0 condition),
the set $C' \cap (\pi \circ \phi)^{-1}(D_\out^o)$
can not contain more than one point.
\end{proof}

\begin{lemma}\label{hphp}
If $f_{a,b} \neq 1$, then there exists a nonconstant map
$$\mathbb{P}^1 \rightarrow X^o_{a,b}[({\bf P}'_a,{\bf P}'_b)]$$
which is both
\begin{enumerate}
\item[(i)] a normalization of a subcurve of $X^o_{a,b}[({\bf P}'_a,{\bf P}'_b)]$,
\item[(ii)]
 an element of $\overline{\foM}(X^o_{a,b}[({\bf P}'_a,{\bf P}'_b)]/D_\out^o)$
where
${\bf P}'_a$ is of size $\text{\em ak}'$ and length $\ell_1$
  and ${\bf P}'_b$
of size $\text{\em bk}'$ and length $\ell_2$.
\end{enumerate}
\end{lemma}

\begin{proof}
Let 
$\mathbb{P}^1\stackrel{\sim}{=}C'\subset C$ be an irreducible component on
which $\pi \circ \phi$ is nonconstant. By Lemmas \ref{jj23} and \ref{jj24},
the map
\begin{equation}\label{u23}
\pi \circ\phi: C' \rightarrow X^o_{a,b}[({\bf P}_a,{\bf P}_b)]
\end{equation}
lies in the moduli space{\footnote{
Since lengths of the partitions match, the spaces 
$X^o_{a,b}[({\bf P}_a,{\bf P}_b)]$ and $X^o_{a,b}[({\bf P}'_a,{\bf P}'_b)]$
 can be taken to be the same.}}
$\overline{\foM}(X^o_{a,b}[({\bf P}'_a,{\bf P}'_b)]/D_\out^o)$
where  ${\bf P}'_a$ is of size $ak'$ and length $\ell_1$
  and ${\bf P}'_b$
of size $bk'$ and length $\ell_2$ for $k'\leq k$.

If \eqref{u23} is birational onto the image
$\pi\circ\phi(C')$, then we have proven the Lemma.
If
$$ \pi \circ\phi: C' \rightarrow 
\pi \circ\phi(C')$$
is a multiple cover, then, by taking the normalization of
 $\pi \circ\phi(C')$, we obtain
the required map
(for $k''<k'$).
\end{proof}

\subsection{Genus inequalities}\label{geni}
On the surface $X_{a,b}$, the intersection results
$$D_1\cdot D_2 =1, \ \ 
D_1 \cdot D_\out = \frac{1}{b}, \ \ \ D_2 \cdot D_\out = \frac{1}{a} $$
are easily obtained   since the divisors intersect transversely (at orbifold points).
Since $A_1(X_{a,b})$ is rank 1 over $\mathbb{Q}$, we conclude
$$ bD_1 = aD_2 = ab D_\out,$$
$$ D_1^2 = \frac{a}{b}, \ \ D_2^2 = \frac{b}{a}, \ \ D_\out^2 = \frac{1}{ab} \ .$$
Since $\beta_k \cdot D_\out = k$, we see $\beta_{k} = abk D_\out$.

The arithmetic genus of a complete
curve $P \subset X^o_{a,b}[({\bf P}_a,{\bf P}_b)]$
of class
$$\beta_k [({\bf P}_a,{\bf P}_b)] = \nu^*(\beta_k) - \sum_{i=1}^{\ell_1} p_i E_i
- \sum_{j=1}^{\ell_2} p_j' E'_j$$
is given by adjunction,
\begin{eqnarray*}
2g_a(P)-2 & = & ( K_{X^o_{a,b}[({\bf P}_a,{\bf P}_b)]}  + P) \cdot P \\
& = & (-D_1-D_2-D_\out + \beta_k) \cdot \beta_k -\sum_{i=1}^{\ell_1} p_i(p_i-1) 
-\sum_{j=1}^{\ell_2} p_j'(p_j'-1) \\
& = &
-ak-bk
-k
+abk^2
-\sum_{i=1}^{\ell_1} p_i(p_i-1) 
-\sum_{j=1}^{\ell_2} p_j'(p_j'-1)\\
& = & ab k^2 -k  -\sum_{i=1}^{\ell_1} p_i^2 -\sum_{j=1}^{\ell_2} (p_j')^2
 \ .
\end{eqnarray*}
If $P$ is irreducible with normalization of genus 0, then
$$ab k^2 -k  -\sum_{i=1}^{\ell_1} p_i^2 -\sum_{j=1}^{\ell_2} (p_j')^2 +2 \geq 0$$
since the arithmetic genus is bounded from below by the
geometric genus.

Suppose $f_{a,b}\neq 1$. By the existence result of Lemma \ref{hphp}, there
exists an irreducible curve $P\subset X^o_{a,b}[({\bf P}_a,{\bf P}_b)]$
with normalization of genus 0. Hence, there exists an integer $k>0$ and
partitions 
\begin{equation}\label{kww3}
\mathbf{P}_a = p_1+\ldots+p_{\ell_1}, \ \  |\mathbf{P}_a|= ak, \ \ \ \
\mathbf{P}_b = p'_1+\ldots+p'_{\ell_2}, \ \ |\mathbf{P}_b|= bk
\end{equation}
for which
the 
inequality
\begin{equation}\label{kww4}
ab k^2 -k  -\sum_{i=1}^{\ell_1} p_i^2 -\sum_{j=1}^{\ell_2} (p_j')^2 +2 \geq 0
\end{equation}
is satisfied.

We define a primitive vector $(a,b)\in \mathbb{Z}^2$
lying strictly in the first quadrant to be {\em permissible} for
the pair $(\ell_1,\ell_2)$ if there 
exist  partitions \eqref{kww3} with $k>0$ satisfying the inequality
\eqref{kww4}. We have proven the following result.

\begin{proposition}
If $f_{a,b}\neq 1$ in the order product factorization of 
 $T_{\ell_2}^{-1} \circ S_{\ell_1} \circ T_{\ell_2}  
\circ S_{\ell_1}^{-1}$, then  $(a,b)$ is permissible for the
pair $(\ell_1,\ell_2)$.
\end{proposition}

\subsection{Case I: Continuous range} \label{bnm1}
Our first result specifies a continuous range of possible slopes
of permissible vectors.
Consider the quadratic polynomial
  $$R_{\ell_1,\ell_2}(z) = \frac{1}{\ell_2} z^2-z + \frac{1}{\ell_1}\ .$$
with discriminant  $1-\frac{4}{\ell_1\ell_2}$.
For the list of pairs
$$(\ell_1,\ell_2) = (1,1),\ (1,2),\ (2,1), \ (1,3),\ (3,1),$$
$R_{\ell_1,\ell_2}(z)>0$ for all real $z$.
For all other pairs of positive integers $(\ell_1,\ell_2)$, the polynomial
$R_{\ell_1,\ell_2}$ has two positive real roots
$$\xi_\pm = \frac{\ell_2}{2} \left(1 \pm \sqrt{1-\frac{4}{\ell_1\ell_2}}\ \right)\ .$$
For slopes
$\xi_-  <   \frac{b}{a} <  \xi_+ $ strictly between the roots, 
$R_{\ell_1,\ell_2}(\frac{b}{a})$ is negative.

\begin{lemma} \label{b334} 
If $R_{\ell_1,\ell_2}(\frac{b}{a})<0$,
then
the vector $(a,b)$ is permissible for $(\ell_1,\ell_2)$.
\end{lemma}

\begin{proof}
If $k$ is chosen to be divisible by both $\ell_1$ and $\ell_2$,
the balanced partitions
$$\mathbf{P}_a = {\frac{ak}{\ell_1}+\ldots+\frac{ak}{\ell_1}},\ \ \ 
\mathbf{P}_b = {\frac{bk}{\ell_2}+\ldots+\frac{bk}{\ell_2}}
$$
can be formed.
The inequality \eqref{kww4} becomes
\begin{equation}\label{kww5}
\left(
ab-\frac{a^2}{\ell_1}-\frac{b^2}{\ell_2}\right) 
k^2 -k  +2 \geq 0
\end{equation}
Since the coefficient of $k^2$ is
$-a^2 R_{\ell_1,\ell_2} (\frac{b}{a})>0$ by the assumed slope
condition, the inequality \eqref{kww5} can certainly
be satisfied for large enough (and divisible) $k$.
\end{proof}

If $(\ell_1,\ell_2)\in \{ (1,4),\ (4,1), \ (2,2) \}$, then the 
polynomial $R_{\ell_1,\ell_2}$ has a double root $\xi_-=\xi_+$.
Lemma \ref{b334} does not permit any slopes in
the double root case. 

\begin{lemma}
If
$(\ell_1,\ell_2) \notin\{ (1,1),\ (1,2),\ (2,1), \ (1,3),\ (3,1), \ (1,4), \ (4,1), 
(2,2)\},$
then the two roots $\xi_\pm$ are real, positive, and irrational.
\end{lemma}
\begin{proof}
Only the irrational claim is nontrivial.
Let $2^s$ be the largest power of $2$ dividing the 
product $\ell_1\ell_2$,
$$\ell_1\ell_2= 2^s n$$
where $n$ is odd.
There are three cases to consider:
\begin{enumerate}
\item[(i)] If $s=0$, 
$$\frac{\ell_1\ell_2-4}{\ell_1\ell_2} = \frac{n-4}{n}$$
where $n-4$ and $n$ are relatively prime.
But there are no positive
pairs of squares separated by $4$, so $\sqrt{1-\frac{4}{\ell_1\ell_2}}$
is irrational.
\item[(ii)] If $s=1$, 
$$\frac{\ell_1\ell_2-4}{\ell_1\ell_2} = \frac{n-2}{n}$$
and the same argument applies.
\item[(iii)] If $s\geq 2$, 
$$\frac{\ell_1\ell_2-4}{\ell_1\ell_2} = \frac{2^{s-2}n-1}{2^{s-2}n}$$
and the argument again applies.
\end{enumerate}
The hypotheses in the Lemma are only used to show
$\ell_1\ell_2-4>0$.
\end{proof}

\begin{lemma} 
If $R_{\ell_1,\ell_2}(\frac{b}{a}) =0$, then we must have
$(\ell_1,\ell_2) \in\{ (1,4),\ (4,1),\ (2,2)\}$. 
Moreover,
 $(a,b)$ is permissible for $(\ell_1,\ell_2)$. \label{b33}
\end{lemma}

\begin{proof}
Since $R_{\ell_1,\ell_2}$ has rational roots only in case
$(\ell_1,\ell_2) \in\{ (1,4),\ (4,1),\ (2,2)\}$,
the first  claim is clear. 
For $(\ell_1,\ell_2)=(1,4)$ and $(4,1)$, 
we have the double roots $(a,b)=(1,2)$ and $(2,1)$ respectively.
For $(\ell_1,\ell_2)=(2,2)$, we have the double root $(a,b)=(1,1)$.
Permissibility is established in both cases by taking $k=2$
and balanced partitions.
\end{proof}

\subsection{Case II: Discrete series} \label{bnm2}
\subsubsection{Positive values}
Permissibility for $R_{\ell_1,\ell_2}(\frac{b}{a}) \leq 0$ has been
established by Lemmas \ref{b334} and \ref{b33}. We now consider the cases
where 
\begin{equation}\label{55t}
R_{\ell_1,\ell_2}\left(\frac{b}{a}\right) > 0\ .
\end{equation}
Since $\sum_{i=1}^{\ell_1} p_i^2\geq \frac{a^2}{\ell_1} k^2$
and similarly for the $p'_j$, we see
\begin{equation*}
ab k^2 -k  -\sum_{i=1}^{\ell_1} p_i^2 -\sum_{j=1}^{\ell_2} (p_j')^2 +2 
\leq  -a^2 R_{\ell_1,\ell_2} \left(\frac{b}{a}\right) k^2 -k +2\ .
\end{equation*}
Certainly for all $k\geq 2$ the right side is negative.
Hence, if $(a,b)$ satisfies \eqref{55t} and is permissible for
$(\ell_1,\ell_2)$, then $k=1$ and we must
have
\begin{equation}
\label{jjb2}
ab   -\sum_{i=1}^{\ell_1} p_i^2 -\sum_{j=1}^{\ell_2} (p_j')^2 +1 =0
\end{equation}
for partitions $p_1+\ldots+p_{\ell_1}= a$ and 
$p'_1+ \ldots +p'_{\ell_2} = b$.

There are exactly three possibilities for the solution of \eqref{jjb2} in the presence of
condition \eqref{55t}:
\begin{enumerate}
\item[(i)] $a\equiv 0$ mod $\ell_1$, $b\equiv 0$ mod $\ell_2$, and 
$a^2R_{\ell_1,\ell_2}\left(\frac{b}{a}\right)=1$.
\item[(ii)] $a\equiv \pm 1$ mod $\ell_1$, $b \equiv 0$ mod $\ell_2$, and 
$a^2R_{\ell_1,\ell_2}\left(\frac{b}{a}\right)= \frac{1}{\ell_1}$,
\item[(iii)] $a\equiv 0$ mod $\ell_1$, $b \equiv \pm 1$ mod $\ell_2$, and
$a^2R_{\ell_1,\ell_2}\left(\frac{b}{a}\right)=\frac{1}{\ell_2}$
.
\end{enumerate}
A straightforward analysis shows unless one of (i-iii) are
satisfied, 
$$ab -\sum_{i=1}^{\ell_1} p_i^2 -\sum_{j=1}^{\ell_2} (p_j')^2 < -a^2R_{\ell_1,\ell_2}\left(\frac{b}{a}\right) 
-1 < -1\ .$$

\subsubsection{Analysis of (i)} 
If $\ell_1$ or $\ell_2$ equals 1, then (i) is special case
of (ii) and (iii). 
Let $\mathcal{S}_{\ell_1,\ell_2}$ be the set of solutions to
(i) with $(a,b)\in \mathbb{Z}^2$ lying in the closed first quadrant.
We will show $\mathcal{S}_{\ell_1,\ell_2}$ is empty when
$\ell_1,\ell_2>1$. 


We now assume $\ell_1,\ell_2>1$.
 When specialized to $b=0$,
the equation  of (i),
\begin{equation}\label{pww33}
a^2R_{\ell_1,\ell_2}\left(\frac{b}{a}\right)=1,
\end{equation}
yields
$\frac{a^2}{\ell_1}= 1$
which has {\em no} solutions satisfying $a\equiv 0 \mod \ell_1$.
A similar conclusion holds when $a=0$.
We conclude all elements of
 $\mathcal{S}_{\ell_1,\ell_2}$
lie strictly in the first
quadrant.

Crucial to our analysis are the following two
 transformations
$$\TT_1(a,b)= (\ell_1 b-a,b), \ \ \ \TT_2(a,b)=(a, \ell_2a-b)\ .$$
which leave the expression
$$a^2R_{\ell_1,\ell_2}\left(\frac{b}{a}\right)=
-ab+\frac{a^2}{\ell_1}+\frac{b^2}{\ell_2} $$
invariant. Both have order two,
$$\TT_1^2 = \TT_2^2 = \text{Id}\ .$$
If $(a,b) \in \mathcal{S}_{\ell_1,\ell_2}$
is a solution of (i) in the first quadrant, we have
seen $a,b>0$. Let
$$(a_1,b_1)= \TT_1(a,b), \ \ \ (a_2,b_2) = \TT_2(a,b)\ .$$
By the invariance, we have
$$a_i^2R_{\ell_1,\ell_2}\left(\frac{b_i}{a_i}\right)= 1$$
for $i=1,2$.
By the definitions of $\TT_i$, the congruence assumptions
for $a$ and $b$ hold also for $a_i$ and $b_i$
respectively.
Since $b_1=b>0$ and
$$ \frac{b^2}{\ell_2} >1,$$
we must have $a_1>0$. Hence, $(a_1,b_1) \in  \mathcal{S}_{\ell_1,\ell_2}$.
Similarly, $(a_2,b_2) \in  \mathcal{S}_{\ell_1,\ell_2}$.
We have proven the following result.

\begin{lemma} Both $\TT_1$ and $\TT_2$ preserve the set 
$\mathcal{S}_{\ell_1,\ell_2}$. \label{pser}
\end{lemma}

We now apply the transformations twice to obtain two
new elements of $\mathcal{S}_{\ell_1,\ell_2}$,
$$(a_{21},b_{21})= \TT_2(a_1,b_1), \ \ \ (a_{12},b_{12}) = \TT_1(a_2,b_2)\ .$$

\begin{lemma} If $(a,b)\in\mathcal{S}_{\ell_1,\ell_2}$
and \label{gg66}
$\frac{b}{a} > \xi_+$, then
$$a>a_{12}, \ \ b>b_{12}, \ \ \frac{b_{12}}{a_{12}} > \frac{b}{a} \ .$$
\end{lemma}

\begin{proof}
Using the formula $a_{12}= \ell_1(\ell_2 a-b)-a$, we find
$a>a_{12}$ is equivalent to
\begin{equation}\label{btt4}
\frac{b}{a} > \ell_2 - \frac{2}{\ell_1} \ .
\end{equation}
But since $\frac{4}{\ell_1\ell_2} \leq 1$, we see
\begin{eqnarray*}
\xi_+ & =& \frac{\ell_2}{2} \left(1 + \sqrt{1-\frac{4}{\ell_1\ell_2}}
\ \right)\\
& \geq & \frac{\ell_2}{2} (1 + 1-\frac{4}{\ell_1\ell_2})\\
& \geq &  \ell_2 - \frac{2}{\ell_1} \ .
\end{eqnarray*}
Hence, inequality \eqref{btt4} follows from the slope assumption
$\frac{b}{a} > \xi_+$.

Similarly, using the formula $b_{12}= \ell_2 a-b$, we find
$b>b_{12}$ is equivalent to
$$\frac{b}{a} > \frac{\ell_2}{2} \ $$
which also follows form the slope assumption.

Since $(a_{12},b_{12}) \in \mathcal{S}_{\ell_1,\ell_2}$, we must have
$a_{12}>0$.
Using the ratio of the formulas for $b_{12}$ and $a_{12}$, we find
$$\frac{b_{12}}{a_{12}} = 
\frac{\ell_2-\frac{b}{a}}
{\ell_1(\ell_2-\frac{b}{a})-1}.$$
The third claim of the Lemma is
$$\frac{\ell_2-\frac{b}{a}}
{\ell_1(\ell_2-\frac{b}{a})-1} > \frac{b}{a}$$
which is equivalent to
$$0 > - R_{\ell_1,\ell_2}\left(\frac{b}{a}\right) = -\frac{1}{a^2}$$
since $(a,b)\in \mathcal{S}_{\ell_1,\ell_2}$.
\end{proof}

\begin{lemma}
\label{gg77}
 If $(a,b)\in\mathcal{S}_{\ell_1,\ell_2}$
and
$\frac{b}{a} < \xi_-$, then
$$a>a_{21}, \ \ b>b_{21}, \ \ \frac{b_{21}}{a_{21}} < \frac{b}{a} \ .$$
\end{lemma}

The proof of Lemma \ref{gg77} is identical to the proof
of Lemma \ref{gg66}.
We are now prepared to prove the emptiness of $\mathcal{S}_{\ell_1,\ell_2}$.

\begin{lemma} For $\ell_1,\ell_2 >1$, we have 
$\mathcal{S}_{\ell_1,\ell_2}=\emptyset$.
\end{lemma}

\begin{proof}
Suppose $(a,b)\in \mathcal{S}_{\ell_1,\ell_2}$ exists.
Then, since
$R_{\ell_1,\ell_2}(\frac{b}{a})>0$, we must have either
$$\frac{b}{a} > \xi_+ \ \ \text{or} \ \ \frac{b}{a} < \xi_- \ .$$
In the former case Lemma \ref{gg66} yields a new
element $(a_{12},b_{12})\in \mathcal{S}_{\ell_1,\ell_2}$
with strictly smaller values $a_{12}<a$ and $b_{12}<b$.
In the latter case, we use Lemma \ref{gg77}.
After finitely many iterations, we must exit the first
quadrant contradicting Lemma \ref{pser}.
\end{proof}

\label{ddf}

\subsubsection{Analysis of (ii)}
We assume $\ell_1,\ell_2>0$ and $(\ell_1,\ell_2)\neq (1,1)$.
 Let $\mathcal{A}_{\ell_1,\ell_2}$ be the set of solutions to
(ii) with $(a,b)\in \mathbb{Z}^2$ lying in the closed first quadrant.
When specialized to $b=0$,
the equation  of (ii),
\begin{equation*}
a^2R_{\ell_1,\ell_2}\left(\frac{b}{a}\right)=\frac{1}{\ell_1},
\end{equation*}
yields
$\frac{a^2}{\ell_1}= \frac{1}{\ell_1}$
which has a
single positive solution $a=1$.
As in Section \ref{ddf},
no solutions occur when $a=0$ (using $(\ell_1,\ell_2)\neq (1,1)$).
We conclude all elements of
 $\mathcal{A}_{\ell_1,\ell_2}$
lie strictly in the first
quadrant except for $(1,0)$.
Let 
$$\mathcal{A}^{*}_{\ell_1,\ell_2} = 
\mathcal{A}_{\ell_1,\ell_2} - \{(1,0)\} \ .$$
The proof of Lemma \ref{pser} immediately yields the following
result.

\begin{lemma} \label{gtgt} Both $\TT_1$ and $\TT_2$ map 
$\mathcal{A}^{*}_{\ell_1,\ell_2}$
to $\mathcal{A}_{\ell_1,\ell_2}$\ .
\end{lemma}

Assume further $(\ell_1,\ell_2) \notin \{(1,1),(1,2),(2,1), (1,3),(3,1) \}$.
The method used in Section \ref{ddf} to study the solutions in case (i)
yields a complete description of 
$\mathcal{A}^{*}_{\ell_1,\ell_2}$. 

\begin{proposition} The permissible vectors for $(\ell_1,\ell_2)$
obtained from case (ii) are
$$\mathcal{A}^{*}_{\ell_1,\ell_2} = \{ \  \TT_2(1,0), \ 
\TT_1(\TT_2(1,0)),\  \TT_2(\TT_1(\TT_2(1,0))),\  \TT_1(\TT_2(\TT_1(\TT_2(1,0)))),\ \ldots \}.$$
\end{proposition}

\begin{proof}
Start with any solution $(a,b) \in \mathcal{A}^{*}_{\ell_1,\ell_2}$.
Depending upon whether $\frac{b}{a}$ is greater than $\xi_+$ or less than 
$\xi_-$ apply  $\TT_1\TT_2$ or $\TT_2\TT_1$. The result
is a solution $(a',b')$ with $a'<a$ and $b'<b$.
By iterating the process, the
solution must eventually 
leave the strict first quadrant. By Lemma \ref{gtgt},
we conclude some chain of applications of $\TT_1$ and $\TT_2$ to $(a,b)$
yields $(1,0)$.
\end{proof}

For the cases $(\ell_1,\ell_2) \in \{(1,1),(1,2),(2,1), (1,3),(3,1) \}$,
the group generated by $\TT_1$ and $\TT_2$ is finite and, in each case,
contains elements that move every $(a,b)$ strictly in the
first quadrant out of the strict first quadrant. Hence, every element
of $\mathcal{A}^{*}_{\ell_1,\ell_2}$ can be reached from $(1,0)$
by a chain of applications of $\TT_1$ and $\TT_2$. Since the sets
are finite, we can list all the elements:
$$
\mathcal{A}^{*}_{1,1} = \{(1,1)\}, \ \
\mathcal{A}^{*}_{1,2} = \{(1,2) \},\ \ \mathcal{A}^{*}_{2,1} = \{(1,1) \},$$
$$\mathcal{A}^{*}_{1,3} = \{ (1,3), \ (2,3)\},\ \
 \mathcal{A}^{*}_{3,1} = \{ (1,1), \ (2,1)\} \ .
$$

\subsubsection{Analysis of (iii)} \label{ddf3}
Of course the discussion of (iii) is identical to (ii).
Let $\mathcal{B}^*_{\ell_1,\ell_2}$ be the set of solutions to
(iii) with $(a,b)\in \mathbb{Z}^2$ lying strictly in the first quadrant.
For
$(\ell_1,\ell_2) \notin \{(1,1),(1,2),(2,1), (1,3),(3,1) \}$,
$$\mathcal{B}^{*}_{\ell_1,\ell_2} = \{ \  \TT_1(0,1), \ 
\TT_2(\TT_1(0,1)),\  \TT_1(\TT_2(\TT_1(0,1))),\  
\TT_2(\TT_1(\TT_2(\TT_1(0,1)))),\ \ldots \}.$$
The special cases are:
$$
\mathcal{B}^{*}_{1,1} = \{(1,1)\}, \ \
\mathcal{B}^{*}_{1,2} = \{(1,1) \},\ \ \mathcal{B}^{*}_{2,1} = \{(2,1) \},$$
$$\mathcal{B}^{*}_{1,3} = \{ (1,1), \ (1,2)\},\ \
 \mathcal{B}^{*}_{3,1} = \{ (3,1), \ (3,2)\} \ .
$$

\subsection{Results for scattering patterns}
Let $\ell_1,\ell_2>0$.
Our main result for scattering patterns determines the set of permissible
vectors for $(\ell_1,\ell_2)$.

\vspace{10pt}
\noindent{\bf Theorem 5.} {\em
If $(\ell_1,\ell_2) \notin \{(1,1),(1,2),(2,1), (1,3),(3,1) \}$, then
the set $\mathcal{P}(\ell_1,\ell_2)$ of permissible vectors 
is the disjoint union}
$$\mathcal{P}_{\ell_1,\ell_2}= \mathcal{A}^*_{\ell_1,\ell_2}\ \cup \ \mathcal{B}^*_{\ell_1,\ell_2}
\ \cup \ \{\ (a,b) \in \mathbb{Z}^2 \ | \ \xi_- \leq \frac{b}{a} \leq \xi_+  \ \}\ .$$ 
\vspace{10pt}

Theorem 5 is simply a summary of the result of Sections \ref{bnm1}-\ref{bnm2}.
The sets of permissible vectors for the special pairs
$(\ell_1,\ell_2)$ excluded in Theorem 5 are:
$$
\mathcal{P}_{1,1} = \{(1,1)\}, \ \
\mathcal{P}_{1,2} = \{(1,2),(1,1) \},\ \ \mathcal{P}_{2,1} = \{(1,1),(2,1) \},$$
$$\mathcal{P}_{1,3} = \{ (1,3),(2,3),(1,1), \ (1,2)\},\ \
 \mathcal{P}_{3,1} = \{ (1,1),(2,1),(3,1), \ (3,2)\} \ .
$$

Returning to Question 4, consider the ordered product 
factorization \eqref{jjtrrrr}
of the commutator. We have proven in Section \ref{geni} the
implication 
$$f_{a,b}\neq 1 \implies
 (a,b) \in \mathcal{P}_{\ell_1,\ell_2}\ . $$
In other words, the scattering pattern associated to $\ell_1$ and $\ell_2$ 
is contained in the directions of $\mathcal{P}_{\ell_1,\ell_2}$.
Theorem 5 completely determines $\mathcal{P}_{\ell_1,\ell_2}$.
In the nontrivial cases $(\ell_1,\ell_2)=(2,2), (3,3)$ and $(2,3)$ analyzed in 
\S\ref{exxc},
the behaviour claimed (via calculations)
fits precisely with the results predicted by Theorem 5.
For $\ell_1=\ell_2=m$, the containment of the scattering pattern in
$\mathcal{P}_{m,m}$ was conjectured previously by Gross-Siebert and Kontsevich
based on computational data.

While very tempting to believe,
we have {\em not} proven the reverse implication
\begin{equation}\label{tempt}
(a,b) \in \mathcal{P}_{\ell_1,\ell_2} \implies f_{a,b}\neq 1.
\end{equation}
Certainly \eqref{tempt} is consistent with all
the gathered data. If $\ell_1=\ell_2=m$, the equivalence
\begin{equation*}
(a,b) \in \mathcal{P}_{m,m} \iff f_{a,b}\neq 1
\end{equation*}
can be proven 
via the existence of
$(1,0)$-semistable representations of the quiver $Q_m$ discussed 
in Section \ref{qq22} below.

\subsection{Quivers}
\label{qq22}
If $\ell_1$ and $\ell_2$ are both equal to $m$, then
Question 4 is related to the existence
of $(1,0)$-semistable representations of $Q_m$ by Theorem 1.

\begin{proposition} For $m=\ell_1=\ell_2$ and 
primitive
$(a,b) \in \mathbb{Z}^2$ lying strictly in the first quadrant,
the following are equivalent:
\begin{enumerate}
\item[(i)] $f_{a,b}\neq 1$,
\item[(ii)] there exists a nonzero $(1,0)$-semistable representation of $Q_m$
with dimension vector proportional to $(a,b)$,
\item[(iii)] there exists a nonzero $(1,0)$-stable back  framed
representation of $Q_m$ with dimension vector
 proportional to $(a,b)$,
\item[(iv)] there exists a nonzero $(1,0)$-stable front  framed
representation of $Q_m$ with dimension vector
 proportional to $(a,b)$,
\end{enumerate}
\end{proposition}

\begin{proof} By Theorem 1, (i) implies 
(iii) and (iv).
The moduli spaces
$\shM^{(1,0),B}_m(d_1,d_2)$ and $\shM^{(1,0),F}_m(d_1,d_2)$
are nonsingular projective varieties with no odd cohomology
\cite{KW,RRR}. For such spaces, nonemptyness
implies positive Euler characteristic.{\footnote{See
\cite{TTT} for better bounds in certain cases.}} 
Hence, again by Theorem 1, (iii) and (iv) are
equivalent and imply (i).
By the definition of $(1,0)$-stability for framed representions,
the underlying standard representation is $(1,0)$-semistable.
So (iii) and (iv) imply (ii).

If (ii) holds, then there exists a $(1,0)$-semistable
representation $\rho$ of $Q_m$ with slope
$$\mu(\rho)= \frac{a}{a+b}\ .$$
We will show there exists a subrepresentation
$\widehat{\rho}\subset \rho$ of the same slope which
is $(1,0)$-stable. If $\rho$ is $(1,0)$-stable, then take
$\widehat{\rho}= \rho$.
If $\rho$ is strictly $(1,0)$-semistable, then $\rho$
must contain a smaller nonzero $(1,0)$-semistable representation 
of slope $\frac{a}{a+b}$, and we repeat.
By finiteness of chains, we must eventually find
a $(1,0)$-stable $\widehat{\rho}$. Since
$$\mu(\widehat{\rho})= \frac{a}{a+b}\ ,$$
the dimension vector of $\widehat{\rho}$ is proportional
to $(a,b)$.
For a $(1,0)$-stable standard representation
$\widehat{\rho} = (\widehat{V}_1, \widehat{V}_2, \tau_1, \ldots, \tau_m)$, 
{\em every}  choice of
 framing data $L_i\subset \widehat{V}_i$ yields a
$(1,0)$-stable framed representation. Hence, (ii)
implies (iii) and (iv).
\end{proof}

Reineke has provided us a proof of 
the following result about
representations of $Q_m$. Given two dimension vectors
${\bf d}=(d_1,d_2)$ and ${\bf e}=(e_1,e_2)$, let
$$\langle {\bf d},{\bf e} \rangle = d_1e_1+d_2e_2 - m d_1e_2.$$
The form $\langle, \rangle$ is {\em not} symmetric.

\vspace{10pt}
\noindent{\bf Proposition 4.15.} (Reineke)
{\em Let ${\bf d}\in \mathbb{Z}^2$ be a primitive vector
lying in the first quadrant.
There exists a $(1,0)$-semistable representation of $Q_m$
with dimension vector proportional to ${\bf d}$ if and only if
$\langle {\bf d},{\bf d}\rangle \leq 1$.}
\vspace{10pt}

\begin{proof}
We start by proving the {\em only if} claim.
Let $\rho$ be a $(1,0)$-semistable representation of $Q_m$
with dimension vector proportional to ${\bf d}$.
We can (as before) assume $\rho$ is $(1,0)$-stable by
passing to a subrepresentation if necessary.
Stability implies $\text{Hom}(\rho,\rho)= \mathbb{C}$. Hence,
$$ \langle{\bf d},{\bf d}\rangle =\text{dim}_{\mathbb{C}}
 \ \text{Hom}(\rho,\rho) -
\text{dim}_{\mathbb{C}} \ \text{Ext}^1(\rho,\rho) \leq 1,$$
where the first
equality is by direct calculation{\footnote{All 
$\text{Ext}^{i\geq 2}(\rho,\rho)$
vanish for quiver representations \cite{R1}.}}.

For the claim in the other direction,
suppose there does not exist a $(1,0)$-semistable
representation with dimension vector ${\bf d}$.
By Corollary 3.5 of \cite{RRR}, there exists a proper{\footnote{By
properness, $s$ is at least 2.}}
decomposition
$$ {\bf d} = {\bf d}^1 +  \ldots + {\bf d}^s$$
into nonzero dimension vectors of $(1,0)$-semistable
representations of $Q_m$ satisfying
$$\mu({\bf d}^1) > \ldots > \mu({\bf d}^s)$$
and $\langle {\bf d}^i,{\bf d}^j\rangle=0$ for all $i<j$.
Let ${\bf e} = {\bf d}^1$ and ${\bf f}= {\bf d}^2 +  \ldots + {\bf d}^s$.
Then,
$${\bf d} = {\bf e} + {\bf f}, \ \ \ \ \mu({\bf e})> \mu({\bf f}), \ \ \ \ 
\langle {\bf e},{\bf f}\rangle=0\ .$$
After writing the last two inequalities as
$$\frac{e_1}{e_2} > \frac{f_1}{f_2}, \ \ \ \ \ \
e_1f_1+e_2f_2 - me_1f_2 = 0$$
and elementary manipulation, we obtain both 
$\langle {\bf e},{\bf e}\rangle>0$
and 
$\langle {\bf f},{\bf f}\rangle>0$.
Moreover,
$$\langle {\bf f},{\bf e}\rangle = e_1f_1+e_2f_2 -me_2f_1 =
m(e_1f_2-e_2f_1) >0.$$
Putting the results together, we conclude
$$\langle {\bf d},{\bf d}\rangle = 
\langle {\bf e},{\bf e}\rangle +
\langle {\bf f},{\bf f}\rangle +\langle {\bf e},{\bf f}\rangle +\langle {\bf f},{\bf e}\rangle \geq 3$$
since all summands are positive except $\langle {\bf e},{\bf f}\rangle=0$.
We have contradicted the assumption
$\langle {\bf d},{\bf d}\rangle \leq 1$.
\end{proof}

For primitive $(a,b) \in \mathbb{Z}^2$ lying strictly
in the first quadrant,
$$a^2 R_{m,m} \left(\frac{b}{a}\right) = \frac{1}{m}\langle (a,b),(a,b)\rangle\ .$$
Proposition 4.15
precisely produces $(1,0)$-semistable representations of $Q_m$
in all the permissible directions. The proof of the claim
\begin{equation}
\label{hhhnn}
(a,b) \in \mathcal{P}_{m,m} \iff f_{a,b}\neq 1
\end{equation}
is complete.
We do not know a proof of \eqref{hhhnn} via rational curve counting
on toric surfaces.

\subsection{Further commutators}
Commutators of more general elements
of the tropical vertex group may be similarly considered. 
Let 
\begin{eqnarray*}
p_1(t,x)& =& 1 + c_1 (tx)^1 + c_2 (tx)^2 + \ldots +c_{\ell_1} (tx)^{\ell_1}, \\
p_2(t,y)& =& 1 + c'_1 (ty)^1 + c'_2 (ty)^2 + \ldots +c'_{\ell_2} (ty)^{\ell_2}
\end{eqnarray*}
be polynomials of degrees $\ell_1$ and $\ell_2$ respectively, and let
$${\mathcal{S}}_{\ell_1}=\theta_{(1,0),p_1(t,x)}, \ \ \ \
{\mathcal{T}}_{\ell_2}=\theta_{(0,1), p_2(t,y)}\ . $$
Our proof of Theorem 5 yields the following result.

\vspace{10pt}
\noindent{\bf Corollary 6.} {\em The scattering pattern associated to
the commutator
\begin{equation*} 
    {\mathcal{T}}_{\ell_2}^{-1} \circ {\mathcal{S}}_{\ell_1} \circ 
{\mathcal{T}}_{\ell_2}
\circ {\mathcal{S}}_{\ell_1}^{-1} =
\stackrel{\rightarrow} \prod \theta_{(a,b),f_{a,b}}\ ,
\end{equation*}
lies in the set $\mathcal{P}_{\ell_1,\ell_2}$.}
\vspace{10pt}

\begin{proof}
By factoring $p_1$ and $p_2$ over $\mathbb{C}$,
 we may instead consider the
scattering pattern associated to 
the commutator of the elements
$${\mathcal{S}}_{\ell_1}=\theta_{(1,0), (1+t_1x)(1+t_2x) \cdots
(1+t_{\ell_1}x)}\ , \ \ \ \ \ \
{\mathcal{T}}_{\ell_2}=\theta_{(0,1), (1+s_1y)(1+s_2y) \cdots
(1+s_{\ell_2}y)}\ $$
in the tropical vertex group over the
ring $\mathbb{C}[[t_1, \ldots, t_{\ell_1}, s_1,\ldots, s_{\ell_2}]]$.
By using the full strength of Theorem 5.4 of \cite{GPS},
the scattering pattern is constrained by the
{\em same} analysis as in Section 4.
\end{proof}

For $\ell_1'\leq \ell_1$ and $\ell_2'\leq \ell_2$,
Corollary 6 suggests the inclusion
$$\mathcal{P}_{\ell'_1,\ell_2'} \subset \mathcal{P}_{\ell_1,\ell_2} \ $$
which can easily be verified directly.
Finally, commutators of the elements
$$\theta_{(v_1,v_2),p_1(t,x^{v_1}y^{v_2})} \ \ \ \text{and} \ \ \ 
\theta_{(w_1,w_2), p_2(t,x^{w_1}y^{w_2})}\  $$
can be transformed to the case constrained by Corollary 6.
We leave the details to the reader.

\section{Symmetry of the scattering diagram}

\subsection{Transformations $\TT_1$ and $\TT_2$}

We return to  the basic elements
$$S_{\ell_1}=\theta_{(1,0),(1+tx)^{\ell_1}} \ \  \text{and} \ \ 
T_{\ell_2}=\theta_{(0,1), (1+ty)^{\ell_2}}\ $$
of the tropical vertex group and
the unique
factorization
\begin{equation}\label{jjtrrrrt}
    T_{\ell_2}^{-1} \circ S_{\ell_1} \circ T_{\ell_2}  
\circ S_{\ell_1}^{-1} =
\stackrel{\rightarrow} \prod \theta_{(a,b),f_{a,b}}\ . 
\end{equation}
We have seen 
$f_{a,b}$ is a series in the variable $(tx)^a(ty)^b$,
$$f_{a,b}(t,x,y) = \mathsf{f}_{a,b}\Big((tx)^a(ty)^b\Big)\  $$
where $\mathsf{f}_{a,b}(z) \in \mathbb{Q}[[z]]$.
By the following result, 
the factorization \eqref{jjtrrrrt}
is symmetric  with respect to the
transformations 
$$\TT_1(a,b)= (\ell_1 b-a,b), \ \ \ \TT_2(a,b)=(a, \ell_2a-b)\ .$$
of Section \ref{ddf}.

\vspace{+10pt}
\noindent {\bf Theorem 7.} {\em Let $(a,b) \in \mathbb{Z}^2$ be a
primitive vector lying strictly in the first quadrant.
If $\TT_1(a,b)$  lies strictly in the first quadrant, then
$$\mathsf{f}_{a,b} = \mathsf{f}_{\TT_1(a,b)}\ .$$
Similarly, if $\TT_2(a,b)$ lies strictly in the first quadrant, then
$\mathsf{f}_{a,b} = \mathsf{f}_{\TT_2(a,b)}$.}
\vspace{10pt}

We will 
prove Theorem 7 in Section \ref{ggtt2} via Theorem 2 and symmetries of Gromov-Witten
invariants of toric surfaces.

\subsection{Curve counting symmetry} \label{ggtt2}

Following the notation of Section \ref{ttss},
let ${\bf P}_a$ and ${\bf P}_b$ be ordered
partitions,
\begin{eqnarray*}
{\bf P}_a&=&p_1+\cdots+p_{\ell_1},\\
{\bf P}_b&=&p_1'+\cdots+p'_{\ell_2},
\end{eqnarray*}
of size $ak$ and $bk$ respectively.
Define partitions ${\bf P}_a'$ and ${\bf P}_b'$
by 
\begin{eqnarray*}
{\bf P}_a'&=&(bk-p_1)+\cdots+(bk-p_{\ell_1}),\\
{\bf P}_b'&=&(ak-p'_1)+\cdots+(ak-p'_{\ell_2}).
\end{eqnarray*}

The following symmetry of Gromov-Witten invariants is the main
step in the proof of Theorem 7.

\begin{proposition} \label{mgmg}
$
N_{a,b}[({\bf P}_a,{\bf P}_b)]=
N_{\ell_1 b-a,b}[({\bf P}_a',{\bf P}_b)]
=N_{a,\ell_2a-b}[({\bf P}_a,{\bf P}_{b}')].
$
\end{proposition}

\begin{proof}
We prove the first equality of Proposition
\ref{mgmg}. The argument for
$$N_{a,b}[({\bf P}_a,{\bf P}_b)]=N_{a,\ell_2a-b}[({\bf P}_a,{\bf P}_{b}')]$$
is, of course, identical.

Consider the surface $Y_{a,b}$ obtained by subdividing
the fan for $X_{a,b}$ by adding a ray in the direction $(1,0)$, as
depicted in Figure \ref{blowupfan}.
Denote by 
$$D_1,\ D_2,\ D_1',\ D_{\out}\subset Y_{a,b}$$
 the divisors corresponding
to the rays generated by $(-1,0)$, $(0,-1)$, $(1,0)$ and $(a,b)$
respectively. Projection onto the second coordinate
induces a map of toric varieties 
$$\pi:Y_{a,b}\rightarrow \PP^1.$$
Both $D_2$ and $D_{\out}$ are fibers of $\pi$, but $D_{\out}$
occurs with multiplicity $b$. Away from $D_{\out}$, $\pi$ is
 $\PP^1$-bundle.
The divisors $D_1$ and $D'_1$ are sections of $\pi$.

Let $Y_{a,b}^o\subset Y_{a,b}$ be the complement of
the four torus fixed points, and let
$$D_i^o = D_i \cap Y_{a,b}^o\ .$$
Choose a set of $\ell_1$ points on $D_1^o$ and a set of
$\ell_2$ points on $D_2^o$.  
Let 
$$
\nu_Y:Y_{a,b}[{\bf P}]\rightarrow
Y_{a,b}\ , \ \ \ \ 
{\bf P}=
({\bf P}_a,{\bf P}_b)\ ,$$ 
be the blow-up along all $\ell_1+\ell_2$ chosen points.
We use the same notation $D_1,D_2,D_1',D_{\out}$ for
the proper transforms in $Y_{a,b}[{\bf P}]$ 
of the respective divisors. 
Let 
$E_1,\ldots,E_{\ell_1}, E_1',\ldots,E_{\ell_2}'$ be the
exceptional divisors of $\nu_Y$.

\begin{figure}
\input{blowupfan.pstex_t}
\caption{}
\label{blowupfan}
\end{figure}

We can similarly consider $\overline{\bf P}=({\bf P}_a',{\bf P}_b)$
and perform the same construction for $(\ell_1b-a,b)$. We obtain
\[
\overline{\nu}_Y: Y_{\ell_1 b-a,b}[\overline{\bf P}]\rightarrow
 Y_{\ell_1 b-a,b}.
\]
Let 
 $\overline{D}_1,\overline{D}_1',\overline{D}_2,
\overline{D}_{\out} \subset Y_{\ell_1 b-a,b} $ be the
toric divisors. We denote their strict transforms with
respect to $\overline{\nu}_Y$
by the same symbols.  
Let
$\overline{E}_1,\ldots,\overline{E}_{\ell_1}, \overline{E}'_1,
\ldots,\overline{E}_{\ell_2}'$ be the
exceptional divisors of $\overline{\nu}_Y$.

Let $x_1,\ldots,x_{\ell_1}\in D_1^o\subseteq
Y_{a,b}$ be the points we have chosen on $D_1^o$.
On $Y_{a,b}[{\bf P}]$,
the proper transforms of
the fibres 
\begin{equation}\label{vtt2}
\pi^{-1}(\pi(x_1)),\ldots,\pi^{-1}(\pi(x_{\ell_1}))
\end{equation}
are $(-1)$-curves linearly equivalent to
$D_2-E_1,\ldots,D_2-E_{\ell_1}$ respectively. 
Let $\eta$ be the blow-down
of  the $\ell_1$ curves \eqref{vtt2} along with
$E_1',\ldots,E_{\ell_2}'$,
\[
\eta: Y_{a,b}[{\bf P}]\rightarrow Z_{a,b}.
\]
The rational map  $\pi\circ \nu_Y \circ\eta^{-1}$
from $Z_{a,b}$ to $\PP^1$ extends
to a morphism 
$$\pi_Z: Z_{a,b}\rightarrow\PP^1$$
with all fibres isomorphic to $\PP^1$ and reduced (except
for the multiple
fibre with support $\eta(D_{\out})$).\footnote{The birational
transformation we have described between the $\PP^1$-bundles
$\pi:Y_{a,b}\rightarrow\PP^1$ and $\pi_Z: Z_{a,b}
\rightarrow\PP^1$ is known as an \emph{elementary transformation}.}
Furthermore, $\eta(D_1)$ 
and $\eta(D_1')$
are sections of $\pi_Z$. From the above geometry, we easily
deduce that $Z_{a,b}$ is a toric variety with
toric boundary
$$\eta(D_1)\cup 
\eta(D_1') \cup \eta(D_2)\cup\eta(D_{\out}).$$ 

Which toric variety is $Z_{a,b}$ ?
 Because the restriction of $\pi_Z$ to
$Z_{a,b}\setminus\eta(D_{\out})$
is a smooth $\PP^1$-bundle over $\AA^1$, we see
$$Z_{a,b}
\setminus\eta(D_{\out})\ \cong \ \PP^1\times\AA^1$$ as toric varieties.
The latter is given, up to lattice isomorphism,
by a fan with rays generated by $(\pm 1,0)$ and $(0,-1)$,
so $Z_{a,b}$ must be given by a fan with an additional ray.
The fan must look exactly like
 Figure \ref{blowupfan}, with $(a,b)$ replaced
by some $(a',b')$:
\begin{enumerate}
\item[$\bullet$]
Since the morphism $\pi_Z$ is induced by projection onto
the second coordinate of the fan and
$\eta(D_{\out})$ is still the
support of a fibre of $\pi_Z$ with multiplicity $b$, we have
$b'=b$.
\item[$\bullet$]
On $Y_{a,b}$, we have $D_1^2=\frac{a}{b}$. Hence,
$D_1^2=\frac{a}{b}-\ell_1$ on 
$Y_{a,b}[{\bf P}]$.
Then, on $Z_{a,b}$,
$$\eta(D_1)^2=\frac{a-\ell_1b}{b}.$$ Thus,
$a'=a-\ell_1b$. 
\end{enumerate}
Using the identification $(a',b')=(a-\ell_1b,b)$,
we conclude
$$Z_{a,b}
\ \cong\ Y_{a-\ell_1b, b} \ \cong \ Y_{\ell_1b-a,b}$$
where the second isomorphism is obtained by
the involution $(m_1,m_2)\mapsto (-m_1,m_2)$
on $\ZZ^2$ identifying the fans for $Y_{a-\ell_1b,b}$  and
$Y_{\ell_1 b-a,b}$.
The composition
$$Y_{a,b}[{\bf P}] \stackrel{\eta}{\rightarrow}
Z_{a,b} \cong Y_{\ell_1b-a,b}$$
is the blow-up of $\ell_1$ points on $\overline{D}_1^o$ and
$\ell_2$ points on $\overline{D}_2^o$.

We have shown, if the point sets for the $\ell_1+\ell_2$  blow-ups 
are chosen
appropriately, there is an isomorphism
$$\varphi: Y_{a,b}[{\bf P}] \stackrel{\sim}{\rightarrow}
Y_{\ell_1b-a,b}[\overline{\bf P}]$$
compatible with boundary geometry
\[
\varphi(D_1)=\overline{D}'_1,\quad 
\varphi(D'_1)=\overline{D}_1,\quad
\varphi(D_2)=\overline{D}_2,\quad
\varphi(D_{\out})=\overline{D}_{\out}.
\]
Let $\beta_k^Y \in H_2(Y_{a,b},\ZZ)$
be the unique class with intersection numbers
\[
\beta^Y_k\cdot D_1=ak,\quad \beta^Y_k\cdot D_1'=0,\quad
\beta^Y_k\cdot D_2=bk,\quad \beta^Y_k\cdot D_{\out}=k \ , 
\]
and let $\beta_k^Y \in H_2(Y_{\ell_1b-a,b},\ZZ)$
be the unique class with intersection numbers
\[
\beta^Y_k\cdot \overline{D}_1=(\ell_1b-a)k,
\quad \beta^Y_k\cdot \overline{D}_1'=0,\quad
\beta^Y_k\cdot \overline{D}_2=bk,\quad 
\beta^Y_k\cdot \overline{D}_{\out}=k \ . 
\]
A straightforward analysis of $\varphi$ yields the
relation
\[
\nu_Y^*(\beta^Y_k)-\sum_{i=1}^{\ell_1}p_i[E_i]
-\sum_{j=1}^{\ell_2}p_j'[E_j']
=
\varphi^*\Big( \overline{\nu}_Y^*(\beta_k^Y)
-\sum_{i=1}^{\ell_1}(bk-p_i)[\overline{E}_i]
-\sum_{j=1}^{\ell_2}p_j'[\overline{E}_j']\Big).
\]

The equality 
$
N_{a,b}[({\bf P}_a,{\bf P}_b)]=
N_{\ell_1 b-a,b}[({\bf P}_a',{\bf P}_b)]$
now follows by unravelling the definitions in Section \ref{mom} of the
Gromov-Witten invariants.
The isomorphism $\varphi$ equates the
corresponding moduli spaces of relative stable maps
$$\overline{\foM}(X^o_{a,b}[{\bf P}]/D_\out^o)\ \cong \
\overline{\foM}(X^o_{\ell_1b-a,b}[\overline{\bf P}]/D_\out^o)\ . $$
The extra blow-ups (corresponding to divisors
$D_1'$ and $\overline{D}_1'$)
occurring in
$Y_{a,b}[{\bf P}]$ and $Y_{\ell_1b-a,b}[\overline{\bf P}]$
 do not
affect the relevant moduli spaces.
\end{proof}

The symmetry of Proposition \ref{mgmg} applied to Theorem 2
immediately yields the symmetry of Theorem 7.  \qed

\vspace{10pt}
If any part of ${\bf P}_a'$ is negative, then 
$N_{\ell_1 b-a,b}[({\bf P}_a',{\bf P}_b)]$ vanishes
since 
$$\overline{\foM}(X^o_{\ell_1b-a,b}[\overline{\bf P}]/D_\out^o)=\emptyset\ .$$
Proposition \ref{mgmg} then asserts the
vanishing of $N_{a,b}[({\bf P}_a,{\bf P}_b)]$.
Similar logic holds if any part of ${\bf P}_b'$ is negative.

Consider the discrete series ${\mathcal B}^*_{\ell_1,\ell_2}$
of the scattering pattern associated to $(\ell_1,\ell_2)$.
By Theorem 7, all the functions ${\mathsf f}_{a,b}$ for
$(a,b) \in {\mathcal B}^*_{\ell_1,\ell_2}$ are {\em equal} to
${\mathsf f}_{\TT_1(0,1)}$. If we apply the transformation
$\TT_1$ to $\TT_1(0,1)$, we leave the strict first quadrant, but
Proposition \ref{mgmg} still 
applies. The result is a simple calculation on
$\mathbb{P}^1\times \mathbb{P}^1$ which we leave to
the reader.

\begin{lemma} In the factorization \eqref{jjtrrrrt},
$$f_{a,b} = \big(1+ (tx)^a(ty)^b\big)^{\ell_2}$$
for all $(a,b) \in {\mathcal B}^*_{\ell_1,\ell_2}$.
\end{lemma}

Similarly, by switching the roles of $x$ and $y$, we obtain the
parallel conclusion for the other discrete series. 

\begin{lemma} In the factorization \eqref{jjtrrrrt},
$$f_{a,b} = \big(1+ (tx)^a(ty)^b\big)^{\ell_1}$$
for all $(a,b) \in {\mathcal A}^*_{\ell_1,\ell_2}$.
\end{lemma}

\subsection{Reflection functors for $Q_m$} In case
$m=\ell_1=\ell_2$,
the symmetry of the factorization \eqref{jjtrrrrt} 
has a very nice interpretation in terms of
the moduli spaces of $(1,0)$-semistable representations of $Q_m$.

Let $\rho=(V_1,V_2, \tau_1, \ldots, \tau_m)$ be a $(1,0)$-semistable
representation of  $Q_m$ with dimension vector $(d_1,d_2)$. 
Consider the canonically associated sequence
\begin{equation}\label{bbyyt}
V_1 \stackrel{\tau}{\rightarrow} \oplus_{i=1}^m V_2 \stackrel{\gamma}{\rightarrow}
\text{Coker}(\tau) \rightarrow 0\ ,
\end{equation}
 where $\tau=(\tau_1, \ldots, \tau_m)$.
The $(1,0)$-semistability condition implies $\tau$ is injective, hence
$$\text{dim}_{\mathbb{C}} \ \text{Coker}(\tau) = md_2 -d_1 \ .$$
The {\em reflection} $R\rho$ is defined to be 
the representation 
$$R\rho=(V_2 , \text{Coker}(\tau), \gamma\circ \iota_1, \ldots, 
\gamma\circ \iota_m)\ ,$$
where $\iota_i$ is the inclusion of $V_2$ as the $i^{th}$ factor
of $\oplus_{i=1}^m V_2$. 
The following Lemma is a standard
result, see  \cite{BBB,TTT}.

\begin{lemma} $R\rho$ is $(1,0)$-semistable. \label{jj233}
\end{lemma}

\begin{proof}
The dimension vector of $R\rho$ is $(d_2,md_2-d_1)$.
Suppose $$U_1\subset V_2 \ \ \text{and} \ \ U_2\subset \text{Coker}(\tau)$$
determine a subrepresentation of $R\rho$ with dimension vector $(u_1,u_2)$.
If $(U_1,U_2)$ destabilizes $R\rho$, then
\begin{equation}\label{hqq2}
\frac{u_1}{u_1+u_2} > \frac{d_2}{(m+1)d_2-d_1} \ .
\end{equation}
An associated subrepresentation of $\rho$ is obtained from the
data
\begin{equation}\label{cfww}
\tau^{-1}(\oplus_{i=1}^m U_1) \subset V_1 \ \ \text{and}\  \ U_1\subset V_2\ .
\end{equation}
Let $u_3$ be the dimension of  $\tau^{-1}(\oplus_{i=1}^m U_1)$.
By sequence \eqref{bbyyt}, 
$u_3 \geq mu_1-u_2$ and hence
\begin{equation}\label{hqqq}
\frac{u_3}{u_3+u_1} \geq \frac{mu_1-u_2}{(m+1)u_1-u_2}\ .
\end{equation}
Using \eqref{hqq2}, we conclude the right side of \eqref{hqqq}
is strictly greater than $\frac{d_1}{d_1+d_2}$. Hence,
the slope of the subrepresentation \eqref{cfww} contradicts the 
$(1,0)$-semistability of $\rho$.
\end{proof}

The inverse to $R$ is defined as follows. From $\rho$,
we construct the sequence
\begin{equation*}
0 \rightarrow \text{Ker}({\tau'}) \stackrel{\gamma'}\rightarrow 
 \oplus_{i=1}^m V_1 \stackrel{{\tau'}}{\rightarrow} V_2 \ ,
\end{equation*}
 where ${\tau}'=\tau_1 \circ \iota'_1+\ldots+\tau_m \circ\iota_m'$
and  $\iota'_i$ is the projection 
of $\oplus_{i=1}^m V_1$ on the $i^{th}$ factor. 
The $(1,0)$-semistability of $\rho$ implies the surjectivity of
$\tau'$. Hence, 
$$\text{dim}_{\mathbb{C}} \ \text{Ker}(\tau') = md_1 -d_2 \ .$$
Define
the representation 
$R^{-1}\rho=(\text{Ker}(\tau'), V_1, \iota'_1\circ \gamma', \ldots, 
\iota'_m\circ \gamma')$.
Following the proof of Lemma \ref{jj233}, we obtain the parallel
result.

\begin{lemma} $R^{-1}\rho$ is $(1,0)$-semistable. \label{jj244}
\end{lemma}

A straightforward verifications shows $R$ and $R^{-1}$ are
inverse to each other,
\begin{equation} \label{56yy}
R^{-1}R \rho \stackrel{\sim}{=} RR^{-1}\rho  
\stackrel{\sim}{=} \rho\ , 
\end{equation}
for all $(1,0)$-semistable representations of $\rho$.
The transformations $R$ and $R^{-1}$ act on dimension vectors by
$$R(a,b)= (b,m b-a), \ \ \ R^{-1}(a,b)=(ma-b,a)\ .$$
Using \eqref{56yy}, we find
isomorphims of moduli spaces
$$\mathcal{M}^{(1,0)}_m(d_1,d_2) \stackrel{\sim}{=}
\mathcal{M}^{(1,0)}_m(R^\pm(d_1,d_2)) $$
for $(d_1,d_2)$ and 
$R^\pm(d_1,d_2)$
in the first quadrant.

Next, we consider the role of the framings of Section \ref{fr11}.
Suppose $\rho$ has a front framing $L_2 \subset V_2$.
The subspace $L_2\subset V_2$ defines  a back
framing for
$R\rho$.
The argument of Lemma
\ref{jj233} yields a refined result.

\begin{lemma} If $(\rho, L_2\subset V_2)$ is a
$(1,0)$-stable front framed representation of $Q_m$, then
$(R\rho, L_2\subset V_2)$ is $(1,0)$-stable back
framed representation.
\end{lemma} 

Similarly, the back framing $L_1\subset V_1$ of $\rho$
determines a front framing of $R^{-1}\rho$.

\begin{lemma} If $(\rho, L_1\subset V_1)$ is a
$(1,0)$-stable back framed representation of $Q_m$, then
$(R^{-1}\rho, L_1\subset V_1)$ is $(1,0)$-stable front
framed representation.
\end{lemma} 

We conclude the reflections yield isomorphisms of
moduli spaces of framed representations as well,{\footnote{The spaces
$\mathcal{M}^{(1,0),B}_m(d_1,d_2)$
and 
$\mathcal{M}^{(1,0),F}_m(R(d_1,d_2))$
may fail to be isomorphic.}}
\begin{equation}\label{zzzop}
\mathcal{M}^{(1,0),F}_m(d_1,d_2)  \stackrel{\sim}{=}
\mathcal{M}^{(1,0),B}_m(R(d_1,d_2)) . 
\end{equation}

For primitive $(a,b)$,
the generating series of Euler characteristics of Section 
\ref{fr22} may be
written as
$$B_{a,b}(t,x,y) = \mathsf{B}_{a,b}\Big( (tx)^a(ty)^b\Big), \ \ \
F_{a,b}(t,x,y) = \mathsf{F}_{a,b}\Big( (tx)^a(ty)^b\Big)\ ,$$
where $\mathsf{B}_{a,b}(z)$  and $\mathsf{F}_{a,b}(z) \in \mathbb{Q}[[z]].$

\begin{proposition}
Let $(a,b)$ be a primitive vector lying strictly in the
first quadrant. If $R(a,b)$ lies in the first quadrant,
$\mathsf{f}_{a,b} = \mathsf{f}_{R(a,b)}$.
\end{proposition}

\begin{proof}
By the isomorphisms \eqref{zzzop} for all
dimension vectors $(ak,bk)$, we conclude
$$\mathsf{F}_{a,b} = \mathsf{B}_{R(a,b)}\ . $$
The result then follows from Theorem 1.
\end{proof}

Since $m=\ell_1=\ell_2$, there is an additional elementary
symmetry given by
\begin{equation}\label{grvv}
\mathsf{f}_{a,b}  =\mathsf{f}_{b,a} \ .
\end{equation}
In the presence of \eqref{grvv}, the symmetry
generated by $R$ is equivalent to the symmetries
generated by $\TT_1$ and $\TT_2$ of Theorem 7.

In the context of the ordered product factorization \eqref{jjtrrrrt}
of the commutator of $S_m$ and $T_m$, the symmetry $R$
was noticed earlier by Kontsevich.

\subsection{Further commutators}
Symmetries of commutators of more general elements
of the tropical vertex group may be similarly considered. 
Let 
\begin{eqnarray*}
p_1(t,x)& =& 1 + c_1 (tx)^1 + c_2 (tx)^2 + \ldots +c_{\ell_1-1} (tx)^{\ell_1-1}
+ (tx)^{\ell_1}, \\
p_2(t,y)& =& 1 + c'_1 (ty)^1 + c'_2 (ty)^2 + \ldots +c'_{\ell_2-1} (ty)^{\ell_2-1} + (ty)^{\ell_2}
\end{eqnarray*}
be polynomials of degrees $\ell_1$ and $\ell_2$ respectively
with highest coefficient equal to 1.
Let
\begin{eqnarray*}
\widehat{p}_1(t,x)& =& 
1 + c_{\ell_1-1} (tx)^1 + c_{\ell_1-2} (tx)^2 + \ldots + c_1 (tx)^{\ell_1-1}+
 (tx)^{\ell_1}, \\
\widehat{p}_2(t,y)& =& 
1 + c'_{\ell_2-1} (ty)^1 + c'_{\ell_2-2} (ty)^2 + \ldots+ c'_1(ty)^{\ell_2-1}
 (ty)^{\ell_2} \ .
\end{eqnarray*}
Consider the four elements
$${\mathcal{S}}_{\ell_1}=\theta_{(1,0),p_1(t,x)}, \ \ \ \
{\mathcal{T}}_{\ell_2}=\theta_{(0,1), p_2(t,y)}\ , $$
$$\widehat{{\mathcal{S}}}_{\ell_1}=\theta_{(1,0),\widehat{p}_1(t,x)}, \ \ \ \
\widehat{\mathcal{T}}_{\ell_2}=\theta_{(0,1), \widehat{p}_2(t,y)}\  $$
of the tropical vertex group.

The scattering pattern associated to
the commutator
\begin{equation*} 
    {\mathcal{T}}_{\ell_2}^{-1} \circ {\mathcal{S}}_{\ell_1} \circ 
{\mathcal{T}}_{\ell_2}
\circ {\mathcal{S}}_{\ell_1}^{-1} =
\stackrel{\rightarrow} \prod \theta_{(a,b),f_{a,b}}\ 
\end{equation*}
is related to the scattering patterns
\begin{equation*} 
    {\mathcal{T}}_{\ell_2}^{-1} \circ \widehat{\mathcal{S}}_{\ell_1} \circ 
{\mathcal{T}}_{\ell_2}
\circ \widehat{\mathcal{S}}_{\ell_1}^{-1} =
\stackrel{\rightarrow} \prod \theta_{(a,b),g_{a,b}} \ \ \ \text{and} \ \ \
   \widehat{\mathcal{T}}_{\ell_2}^{-1} \circ {\mathcal{S}}_{\ell_1} \circ 
\widehat{\mathcal{T}}_{\ell_2}
\circ {\mathcal{S}}_{\ell_1}^{-1} =
\stackrel{\rightarrow} \prod \theta_{(a,b),h_{a,b}}\ . \ \ 
\end{equation*}
As before, let
$$f_{a,b}(t,x,y) = \mathsf{f}_{a,b}\Big((tx)^a(ty)^b\Big), \ \ \ \ \
g_{a,b}(t,x,y) = \mathsf{g}_{a,b}\Big((tx)^a(ty)^b\Big), $$
$$h_{a,b}(t,x,y) = \mathsf{h}_{a,b}\Big((tx)^a(ty)^b\Big).
  $$

\vspace{+10pt}
\noindent {\bf Corollary 8.} {\em Let $(a,b) \in \mathbb{Z}^2$ be a
primitive vector lying strictly in the first quadrant.
If $\TT_1(a,b)$  lies strictly in the first quadrant, then
$$\mathsf{f}_{a,b} = \mathsf{g}_{\TT_1(a,b)}\ .$$
Similarly, if $\TT_2(a,b)$ lies strictly in the first quadrant, then
$\mathsf{f}_{a,b} = \mathsf{h}_{\TT_2(a,b)}$.}
\vspace{10pt}

\begin{proof}
As in the proof of Corollary 6, we start by 
factoring $p_1$ and $p_2$ over $\mathbb{C}$,
 $${\mathcal{S}}_{\ell_1}=\theta_{(1,0), (1+t_1x)(1+t_2x) \cdots
(1+t_{\ell_1}x)}\ , \ \ \ \ \ \
{\mathcal{T}}_{\ell_2}=\theta_{(0,1), (1+s_1y)(1+s_2y) \cdots
(1+s_{\ell_2}y)}\ .$$
The result then follows from 
Proposition \ref{mgmg} applied to
Theorem 5.4 of \cite{GPS}.
\end{proof}

\section{Further directions} \label{furr}
There are several interesting questions in the subject
which we have not been able to discuss here. We end by stating 
three:
\begin{enumerate}
\item[(i)]
The functions $f_{a,b}$ associated to
the commutator \eqref{jjtrrrr}
should satisfy certain 
integrality properties. In the $\ell_1=\ell_2$ case,
the relevant integrality is conjectured by Kontsevich and
Soibelman in \cite{ks2} 
and proven by Reineke in \cite{R3}. The
integrality of Conjecture 6.2 of \cite{GPS} constrains
all cases $(\ell_1,\ell_2)$ and, more generally, genus 0
relative Gromov-Witten 
invariants of surfaces (where the curves
have full contact order at a single point with the relative
divisor). Conjecture 6.2 of \cite{GPS} remains open.

\item[(ii)] The curve counting side of Corollary 3 has
a very natural higher genus extension discussed in
Section 5.8 of \cite{GPS} involving the 
top Chern class $\lambda_g$ of the Hodge bundle on
$\overline{M}_g$. The quiver side of Corollary 3 has
a natural extension by replacing the Euler characteristic
with the Poincar\'e polynomial. The two extensions
do not naively match. What is the meaning of the
higher genus Gromov-Witten theory on the quiver side?

\item[(iii)] Let $m$ be fixed.
M. Douglas  has conjectured 
the function
$$\frac{1}{a}\log\left(\chi\big(\mathcal{M}^{(1,0)}_m(a,b)\big)\right)$$
asymptotically (for large and primitive $(a,b)$) 
depends only upon $\frac{b}{a}$.
 Moreover, the limit function should be continuous. See
\cite{TTT} for a discussion of results toward the conjecture.
\end{enumerate}

A physical context for studying $m$-Kronecker quivers 
is explained in Section 4 of \cite{DeMo}.
Prediction (iii) fits naturally in the framework of \cite{DeMo}.

\section*{Acknowledgments}
Sections 1-3 of the paper closely follow a three lecture
course given by R.P. at the {\em Geometry summer
school} in July 2009 at the Instituto Superior T\'ecnico in Lisbon.
The results of Section 4 and 5 were worked out together during the IST visit and after. 
We thank the organizers A. Cannas da Silva
and R. Fernandes for providing a stimulating environment.

Conversations with T. Bridgeland and A. King about quivers
at a workshop on
{\em Derived categories} at the Simons center in Stony Brook
in January 2009
played an important role in the presentation of Section 2.
We are very grateful to M. Reineke for contributing 
Proposition 4.15 on the existence of semistables.
For the analysis in Sections \ref{ddf}-\ref{ddf3},
M. Bhargava has found an alternate argument via the genus theory
of Gauss for values of integral quadratic forms.
We thank
P. Hacking,  
S. Keel,
M. 
Kontsevich, D. Maulik, I. Setayesh, B. Siebert, Y. Soibelman, 
 and
R. Thomas 
for many related discussions.

M.~G. was partially supported by NSF grant DMS-0805328 and
the DFG.
R.~P. was partially supported by NSF grant DMS-0500187.


\end{document}

%% file: diagram11.pstex_t
\begin{picture}(0,0)%
\includegraphics{diagram11.pstex}%
\end{picture}%
\setlength{\unitlength}{3947sp}%
\begingroup\makeatletter\ifx\SetFigFont\undefined%
\gdef\SetFigFont#1#2#3#4#5{%
  \reset@font\fontsize{#1}{#2pt}%
  \fontfamily{#3}\fontseries{#4}\fontshape{#5}%
  \selectfont}%
\fi\endgroup%
\begin{picture}(2727,2634)(1189,-2173)
\put(3901,314){\makebox(0,0)[lb]{\smash{{\SetFigFont{12}{14.4}{\rmdefault}{\mddefault}{\updefault}{\color[rgb]{0,0,0}$1+t^2xy$}%
}}}}
\end{picture}%

%% file: diagram22.pstex_t
\begin{picture}(0,0)%
\includegraphics{diagram22.pstex}%
\end{picture}%
\setlength{\unitlength}{3947sp}%
\begingroup\makeatletter\ifx\SetFigFont\undefined%
\gdef\SetFigFont#1#2#3#4#5{%
  \reset@font\fontsize{#1}{#2pt}%
  \fontfamily{#3}\fontseries{#4}\fontshape{#5}%
  \selectfont}%
\fi\endgroup%
\begin{picture}(3024,3024)(1189,-2173)
\put(3151,239){\makebox(0,0)[lb]{\smash{{\SetFigFont{12}{14.4}{\rmdefault}{\mddefault}{\updefault}{\color[rgb]{0,0,0}$\hdots$}%
}}}}
\put(3601,-61){\makebox(0,0)[lb]{\smash{{\SetFigFont{12}{14.4}{\rmdefault}{\mddefault}{\updefault}{\color[rgb]{0,0,0}$\vdots$}%
}}}}
\end{picture}%

%% file: diagram33.pstex_t
\begin{picture}(0,0)%
\includegraphics{diagram33.pstex}%
\end{picture}%
\setlength{\unitlength}{3947sp}%
\begingroup\makeatletter\ifx\SetFigFont\undefined%
\gdef\SetFigFont#1#2#3#4#5{%
  \reset@font\fontsize{#1}{#2pt}%
  \fontfamily{#3}\fontseries{#4}\fontshape{#5}%
  \selectfont}%
\fi\endgroup%
\begin{picture}(2724,2724)(1189,-2173)
\put(2176,164){\makebox(0,0)[lb]{\smash{{\SetFigFont{12}{14.4}{\rmdefault}{\mddefault}{\updefault}{\color[rgb]{0,0,0}$\cdots$}%
}}}}
\put(3451,-1121){\makebox(0,0)[lb]{\smash{{\SetFigFont{12}{14.4}{\rmdefault}{\mddefault}{\updefault}{\color[rgb]{0,0,0}$\vdots$}%
}}}}
\end{picture}%

%% file: P2abfan.pstex_t
\begin{picture}(0,0)%
\includegraphics{P2abfan.pstex}%
\end{picture}%
\setlength{\unitlength}{3947sp}%
\begingroup\makeatletter\ifx\SetFigFont\undefined%
\gdef\SetFigFont#1#2#3#4#5{%
  \reset@font\fontsize{#1}{#2pt}%
  \fontfamily{#3}\fontseries{#4}\fontshape{#5}%
  \selectfont}%
\fi\endgroup%
\begin{picture}(2002,2436)(4611,-5505)
\put(6551,-3216){\makebox(0,0)[lb]{\smash{{\SetFigFont{12}{14.4}{\rmdefault}{\mddefault}{\updefault}{\color[rgb]{0,0,0}$(a,b)$}%
}}}}
\put(5811,-5441){\makebox(0,0)[lb]{\smash{{\SetFigFont{12}{14.4}{\rmdefault}{\mddefault}{\updefault}{\color[rgb]{0,0,0}$(0,-1)$}%
}}}}
\put(4626,-4606){\makebox(0,0)[lb]{\smash{{\SetFigFont{12}{14.4}{\rmdefault}{\mddefault}{\updefault}{\color[rgb]{0,0,0}$(-1,0)$}%
}}}}
\end{picture}%

%% file: blowupfan.pstex_t
\begin{picture}(0,0)%
\includegraphics{blowupfan.pstex}%
\end{picture}%
\setlength{\unitlength}{4144sp}%
\begingroup\makeatletter\ifx\SetFigFontNFSS\undefined%
\gdef\SetFigFontNFSS#1#2#3#4#5{%
  \reset@font\fontsize{#1}{#2pt}%
  \fontfamily{#3}\fontseries{#4}\fontshape{#5}%
  \selectfont}%
\fi\endgroup%
\begin{picture}(2720,3008)(76,-2260)
\put(2651,589){\makebox(0,0)[lb]{\smash{{\SetFigFontNFSS{12}{14.4}{\rmdefault}{\mddefault}{\updefault}{\color[rgb]{0,0,0}$(a,b)$}%
}}}}
\put(1496,-2191){\makebox(0,0)[lb]{\smash{{\SetFigFontNFSS{12}{14.4}{\rmdefault}{\mddefault}{\updefault}{\color[rgb]{0,0,0}$(0,-1)$}%
}}}}
\put( 91,-1091){\makebox(0,0)[lb]{\smash{{\SetFigFontNFSS{12}{14.4}{\rmdefault}{\mddefault}{\updefault}{\color[rgb]{0,0,0}$(-1,0)$}%
}}}}
\put(2781,-1106){\makebox(0,0)[lb]{\smash{{\SetFigFontNFSS{12}{14.4}{\rmdefault}{\mddefault}{\updefault}{\color[rgb]{0,0,0}$(1,0)$}%
}}}}
\end{picture}%